\pdfoutput=1
\RequirePackage{ifpdf}
\ifpdf 
\documentclass[pdftex]{sigma}
\else
\documentclass{sigma}
\fi

\numberwithin{equation}{section}

\newtheorem{Theorem}{Theorem}[section]
\newtheorem{Corollary}[Theorem]{Corollary}
\newtheorem{Lemma}[Theorem]{Lemma}
\newtheorem{Proposition}[Theorem]{Proposition}
\newtheorem{cla}[Theorem]{Claim}
 { \theoremstyle{definition}
\newtheorem{Definition}[Theorem]{Definition}
\newtheorem{Observation}[Theorem]{Observation}
\newtheorem*{Main}{Main result}
\newtheorem{Remark}[Theorem]{Remark} }

\usepackage[all]{xy}
\usepackage{bm}

\begin{document}

\global\long\def\cvarphi{\check{\varphi}}
\global\long\def\crho{\check{\rho}}
\global\long\def\rG{\mathrm{G}}
\global\long\def\sE{\mathcal{E}}
\global\long\def\sK{\mathcal{K}}
\global\long\def\sO{\mathcal{O}}
\global\long\def\mC{\mathbb{C}}
\global\long\def\mP{\mathbb{P}}
\global\long\def\mR{\mathbb{R}}
\global\long\def\mQ{\mathbb{Q}}
\global\long\def\mZ{\mathbb{Z}}
\global\long\def\Hom{\mathrm{Hom}}
\global\long\def\sZ{\mathcal{Z}}

\allowdisplaybreaks

\newcommand{\arXivNumber}{1707.08728}

\renewcommand{\thefootnote}{}

\renewcommand{\PaperNumber}{039}

\FirstPageHeading

\ShortArticleName{Movable vs Monodromy Nilpotent Cones of Calabi--Yau Manifolds}

\ArticleName{Movable vs Monodromy Nilpotent Cones\\ of Calabi--Yau Manifolds\footnote{This paper is a~contribution to the Special Issue on Modular Forms and String Theory in honor of Noriko Yui. The full collection is available at \href{http://www.emis.de/journals/SIGMA/modular-forms.html}{http://www.emis.de/journals/SIGMA/modular-forms.html}}}

\Author{Shinobu HOSONO and Hiromichi TAKAGI}

\AuthorNameForHeading{S.~Hosono and H.~Takagi}

\Address{Department of Mathematics, Gakushuin University, Mejiro, Toshima-ku, Tokyo 171-8588, Japan}
\Email{\href{mailto:hosono@math.gakushuin.ac.jp}{hosono@math.gakushuin.ac.jp}, \href{mailto:hiromici@math.gakushuin.ac.jp}{hiromici@math.gakushuin.ac.jp}}

\ArticleDates{Received September 11, 2017, in final form April 23, 2018; Published online May 02, 2018}

\Abstract{We study mirror symmetry of complete intersection Calabi--Yau manifolds which have birational automorphisms of infinite order. We observe
that movable cones in birational geometry are transformed, under mirror symmetry, to the monodromy nilpotent cones which are naturally glued together.}

\Keywords{Calabi--Yau manifolds; mirror symmetry; birational geometry; Hodge theory}

\Classification{14E05; 14E07; 14J33; 14N33}

\renewcommand{\thefootnote}{\arabic{footnote}}
\setcounter{footnote}{0}

\section{Introduction}

A smooth projective variety $X$ of dimension $n$ is called a Calabi--Yau $n$-fold if the canonical bundle $K_{X}$ is trivial and $H^{i}(X,\mathcal{O}_{X})=0$, $1\leq i\leq n-1$. In the 90's, the idea of mirror symmetry was discovered in theoretical physics and has long been a source of many mathematical ideas related to Calabi--Yau manifolds. After more than 20 years since its discovery, we have now several approaches \cite{GS1,GS2, Ko,SYZ} toward mathematical understanding of the symmetry.

In this paper, we will focus on ``classical'' mirror symmetry of Calabi--Yau threefolds, i.e., we compare two different moduli spaces associated to Calabi--Yau threefolds, the K\"ahler moduli and the complex structure moduli spaces, considering Calabi--Yau threefolds which have several birational models. According to birational geometry of higher dimensional manifolds, if a Calabi--Yau threefold $X$ has birational models, then the K\"ahler cone of $X$ can be extended to the movable cone $\operatorname{Mov}(X)$ \cite{KawCone, MoKahler}. On the mirror side, corresponding to each birational model, there appears a special boundary point called large complex structure limit, which is characterized by unipotent monodromy \cite{MoLCSL}. Using this unipotent property, the so-called monodromy nilpotent cone is defined for each boundary point. We will find that, as a result of monodromy relations, the monodromy nilpotent cones glue together to define a~larger cone which can be identified with the movable cone $\operatorname{Mov}(X)$ under mirror symmetry.

Studying birational geometry in mirror symmetry (or string theory) goes back to papers by Morrison and his collaborators in the 90's~\cite{AGM}. The birational geometry discussed in the 90's was mostly for Calabi--Yau hypersurfaces in toric varieties, and it comes from the different resolutions of ambient toric varieties. In this paper, we will study two specific examples of complete intersection Calabi--Yau threefolds for which we have birational models in slightly different form, and also have birational automorphisms of infinite order.

The construction of this paper is as follows: In Section~\ref{sec:Classical-MS}, we will first recall some background material on mirror symmetry as formulated in the 90's. Restricting our attentions to three dimensional Calabi--Yau manifolds, we will summarize the basic properties of Calabi--Yau manifolds called A- and B-structures. In Section~\ref{sec: CICY-I}, we will introduce a specific Calabi--Yau threefold given by a complete intersection in $\mP^{4}\times\mP^{4}$, whose birational geometry and mirror symmetry were studied in detail in previous works~\cite{HTcmp, HTjag}. We will describe its movable cone by studying the geometry of birational models. In Section~\ref{sec:Gluing-monod-I}, we will report some results of monodromy calculations, and describe the details of how the monodromy nilpotent cones glue together by monodromy relations. In Section~\ref{section5}, we will present another complete intersection given in $\mP^{3}\times\mP^{3}$. Although there do not appear other birational models to this Calabi--Yau threefold than itself, we will observe interesting gluing property of monodromy nilpotent cones which corresponds to the structure of the movable cone observed in~\cite{Oguiso1}. Summary and discussions will be presented in Section~\ref{section6}. There we will also describe the corresponding calculations for a~K3 surface in $\mP^{3}\times\mP^{3}$ which has a parallel description to the complete intersection in $\mP^{4}\times\mP^{4}$.

\section{Classical mirror symmetry} \label{sec:Classical-MS}

\subsection{Mirror symmetry of Calabi--Yau threefolds}

Let us consider Calabi--Yau threefolds $X$ and $X^{*}$ which will be taken to be mirror to each other. For each of these, we have two different structures, called A-structure and B-structure.

\subsubsection[A-structure of $X$]{A-structure of $\boldsymbol{X}$} Let $\mathcal{K}_{X}$ be the K\"ahler cone of $X$ and $\kappa_{1},\dots,\kappa_{r}\in H^{1,1}(X,\mR)=H^{1,1}(X)\cap H^{2}(X,\mR)$ be ge\-nerators of the K\"ahler cone, where for simplicity, we assume that the K\"ahler cone is a simplicial cone in $H^{2}(X,\mathbb{R})$. Let $\kappa$ be the K\"ahler class which corresponds to the polarization of $X$ and write $\kappa$ by
\begin{gather*}
\kappa=t_{1}\kappa_{1}+\cdots+t_{r}\kappa_{r},
\end{gather*}
with $t_{i}>0$. The Lefschetz operator $L_{\kappa}(-):=\kappa\wedge(-)$ defines a nilpotent linear action on the even cohomology $H^{\rm even}(X):=\oplus_{p}H^{p,p}(X)$. In fact, this is a part of the Lefschetz $\mathfrak{sl}(2,\mathbb{C})$ action, and defines the following decomposition:
\begin{gather*}
\begin{matrix}H^{0,0}\\
\\
H^{1,1}\\
\\
H^{2,2}\\
\\
H^{3,3}
\end{matrix}\;\;\;\;=\;\;\;\;\begin{matrix}\bullet\\
\downarrow\\
\bullet & \bullet & & \bullet~~~~~\\
\downarrow & \downarrow & \cdots & \downarrow L_{\kappa}.\\
\bullet & \bullet & & \bullet~~~~~ \\
\downarrow\\
\bullet
\end{matrix}
\end{gather*}

From the viewpoint of homological mirror symmetry, it is natural to replace $H^{\rm even}(X)$ with the Grothendieck group $K(X)$ (modulo torsion) which is an abelian group equipped the symplectic form
\begin{gather*}
\chi(-,-)\colon \ K(X)\times K(X)\to\mathbb{Z}
\end{gather*}
with $\chi$ defined by $\chi(\mathcal{E},\mathcal{F}):=\sum(-1)^{i}\dim H^{i}(X,\mathcal{E}^{*}\otimes\mathcal{F})$ for vector bundles. Based on this integral and symplectic structure on $K(X)$, we can introduce the corresponding structure on~$H^{\rm even}(X,\mathbb{Q})$. \textit{A-structure of $X$} is the nilpotent action $L_{\kappa}$ on $H^{\rm even}(X,\mathbb{Q})$ with this integral and symplectic structure.

\subsubsection[B-Structure of $X^*$]{B-Structure of $\boldsymbol{X^*}$} Let $X^{*}=X_{b_{0}}^{*}$ and consider
a smooth deformation family $\pi\colon \mathfrak{X}^{*}\to B$ of $X_{b_{0}}^{*}(b_{0}\in B)$ over some open parameter space $B$. We denote by $X_{b}^{*}=\pi^{-1}(b)$ the fiber over $b\in B$. Then we have Kodaira--Spencer map $\rho_{b}\colon T_{b}B\to H^{1}(X_{b},\mathcal{T}X_{b}^{*})$
which we assume to be an isomorphism. Associated to this family, we naturally have the local system $R^{3}\pi_{*}\mathbb{C}_{\mathfrak{X}^{*}}$
on $B$. In the 90's, mirror symmetry was recognized by finding some local family $\mathfrak{X}|_{\Delta_{r}^{*}}\to\Delta_{r}^{*}$ with special
properties over the product of punctured disc $\Delta_{r}^{*}=(\Delta^{*})^{r}$ where $\Delta^{*}=\left\{ z\in\mathbb{C}\,|\,0<|z|<1\right\} $ and
$\dim B=r$. The required properties for the local family are described by the monodromy representation of the fundamental group $\pi_{1}(\Delta_{r}^{*})\simeq\mathbb{Z}^{r}$ for the local system $R^{3}\pi_{*}\mathbb{C}_{\mathfrak{X}^{*}}$ restricted to over $\Delta_{r}^{*}$. Let $T_{i}$ represent the monodromy matrix corresponding to the $i$-th generator of $\pi_{1}(\Delta_{r}^{*})$ with fixing a base point $b_{0}\in\Delta_{r}^{*}$. Assuming that all $T_{i}$ are unipotent, we have nilpotent matrices $N_{i}=\log T_{i}=\sum\limits_{k\geq1}\frac{(-1)^{k-1}}{k}(T_{i}-{\rm id})^{k}$. The set
\begin{gather}
\Sigma=\left\{ \sum\lambda_{i}N_{i}\,|\,\lambda_{i}\in\mathbb{R}_{>0}\right\} \label{eq:Mon-cone-Sig}
\end{gather}
is called \textit{monodromy nilpotent cone} consisting of nilpotent matrices on $H^{3}(X_{b_{0}},\mathbb{\mathbb{Q}})$. It is known that each element of $\Sigma$ defines the same monodromy weight filtration on $H^{3}(X_{b_{0}},\mathbb{Q})$ (see \cite[Theorem~1.9]{Griffiths}). The following definition is due to Morrison~\cite{MoLCSL}.

\begin{Definition} \label{def:LCSL}The degeneration of the local family $\mathfrak{X}|_{\Delta_{r}^{*}}\to\Delta_{r}^{*}$ at the origin is called a large complex structure limit (LCSL) if the following hold:
\begin{enumerate}\itemsep=0pt
\item[(1)] All $T_{i}$, $i=1,\dots,r$, are unipotent.

\item[(2)] Let $N_{\lambda}=\sum_{i}\lambda_{i}N_{i}$ $\lambda_{i}>0$. This induces the monodromy weight filtration,
\begin{gather}
\begin{matrix}W_{0}=W_{1} & \subset & W_{2}=W_{3} & \subset & W_{4}=W_{5} & \subset & W_{6} & =H^{3}(X_{b_{0}},\mathbb{Q})\\
\bullet & \leftarrow & \bullet & \leftarrow & \bullet & \leftarrow & \bullet\\
 & & \bullet & \leftarrow & \bullet\\
 & & \vdots & & \vdots\\
 & & \bullet & \leftarrow & \bullet
\end{matrix}\label{eq:mw-filtration}
\end{gather}
with $\dim W_{0}=1$ and $\dim W_{2}=1+r$.

\item[(3)] Let $W_{0}=\mathbb{Q}w_{0}$ and introduce a bi-linear form on $W_{0}$ by $\langle w_{0},w_{0}\rangle=1$. This defines $m_{jk}:=\langle w_{0},N_{j}w_{k}\rangle$ for a $\mathbb{Q}$-basis $[w_{1}],\dots,[w_{r}]$ of $W_{2}/W_{0}$. Then the $r\times r$ matrix $(m_{jk})_{1\leq j,k\leq r}$ is an invertible $\mathbb{Q}$-matrix.
\end{enumerate}
\end{Definition}

We note that there is a natural integral symplectic structure on $H^{3}(X_{b_{0}}^{*},\mathbb{Z})$, and the mono\-dromy matrices $T_{i}$ are given by integral and symplectic matrices if we fix a symplectic basis of $H^{3}(X_{b_{0}}^{*},\mathbb{Z})$. \textit{B-structure of} $X^{*}$ at LCSL is defined to be such an integral and symplectic basis of $H^{3}(X_{b_{0}}^{*},\mathbb{Z})$ with the monodromy matrices $T_{i}$ which are compatible with the filtration~(\ref{eq:mw-filtration}).

\subsubsection{Mirror symmetry} In classical mirror symmetry, $X$ is called a mirror to $X^{*}$ if the A-structure of $X$ is isomorphic to the
B-structure of $X^{*}$, i.e., the two nilpotent actions $L_{\kappa}$ and $N_{\lambda}$ are identified together with their integral and symplectic structures. To be more explicit, suppose we have a B-structure at a LCSL. Since $N_{\lambda}^{4}=0$ and $N_{\lambda}^{3}W_{6}\subset W_{0}$, we have
\begin{gather*}
N_{i}N_{j}N_{k}=C_{ijk}\mathtt{N}_{0}
\end{gather*}
with a fixed rank one nilpotent matrix $\mathtt{N}_{0}$ satisfying $N_{i}\mathtt{N}_{0}=0$. Corresponding to this products of nilpotent matrices, we have, in the A-structure, the cup-product
\begin{gather*}
\kappa_{i}\cup\kappa_{j}\cup\kappa_{k}=K_{ijk}\mathtt{V}_{0}, \qquad K_{ijk}\in\mathbb{Z},
\end{gather*}
where $\mathtt{V}_{0}\in H^{3,3}(X)$ normalized by $\int_{X}\mathtt{V}_{0}=1$. After fixing a normalization of the matrix $\mathtt{N}_{0}$, we have $C_{ijk}=K_{ijk}$ if $X$ and $X^{*}$ are mirror to each other, in particular we have $C_{ijk}\in\mathbb{Z}_{\geq0}$. In fact, $C_{ijk}$ is the leading coefficient of the so-called Griffiths--Yukawa coupling, and $K_{ijk}$ is the leading term of the quantum product. Mirror symmetry implies the equality between the two in full orders under the so-called \textit{mirror map}.

\subsection{Birational geometry and mirror symmetry}

Calabi--Yau threefolds often come with birational models. Mirror symmetry in such cases has been studied in~\cite{MoKahler} and is known as topology change in physics~\cite{AGM}. The purpose of this paper is to elaborate such cases in more details comparing the A-structure of~$X$ and the B-structure of~$X^{*}$. In the 90's, Morrison considered the movable cone of~$X$ in the context of mirror symmetry and also the topology change. We will push this perspective further by finding the corresponding cone structure in terms of the monodromy nilpotent cones in the B-structure of~$X^{*}$.

\subsubsection[Movable cones of $X$]{Movable cones of $\boldsymbol{X}$} As above, let us assume that Calabi--Yau threefold $X=:X_{1}$ comes with several other Calabi--Yau threefolds~$X_{i}$, $i=2,\dots,s$, which are birational to each other. Let $\mathcal{K}_{i}\subset H^{2}(X_{i},\mathbb{R})$ be the K\"ahler cone of~$X_{i}$. Using the birational maps $\varphi_{i}\colon X\dashrightarrow X_{i}$, these K\"ahler cones of~$X_{i}$ can be transformed to the corresponding cones in~$H^{2}(X,\mathbb{R})$. The convex hull of the union of these cones is the movable cone~$\operatorname{Mov}(X)$ of~$X$. It is shown in~\cite{KawMovC} that the union of the transformed K\"ahler cones defines a chamber structure to the movable cone $\operatorname{Mov}(X)$ (see also \cite[Section~5]{MoKahler}). To work with the classical mirror symmetry, in fact, we have to consider the movable cone in~$H^{2}(X,\mathbb{R})\otimes\mathbb{C}$ using the complexified K\"ahler cones $\mathcal{K}_{i}+\sqrt{-1}H^{2}(X_{i},\mathbb{R})$. However, in this paper, we will mostly focus on the structures in the real part of the complexified K\"ahler moduli.

\subsubsection[Compactification of the moduli space ${\mathcal M}_{X^*}^{\rm cpx}$]{Compactification of the moduli space $\boldsymbol{{\mathcal M}_{X^*}^{\rm cpx}}$}\label{para:obs-path}

Suppose that $X^{*}$ is mirror to $X$, i.e., we have a mirror family $\mathfrak{X}^{*}\to B:=\mbox{\ensuremath{\mathcal{M}}}_{X^{*}}^{\rm cpx}$ over a parameter space $\mathcal{M}_{X^{*}}^{\rm cpx}$ on which we find a local (smooth) family $\mathfrak{X}|_{\Delta_{r}^{*}}\to\Delta_{r}^{*}\subset\mathcal{M}_{X^{*}}^{\rm cpx}$ to describe the B-structure which is mirror to the A-structure of~$X$. In the classical mirror symmetry of Calabi--Yau complete intersections in toric varieties, there is a natural (toric) compactification $\overline{\mathcal{M}}_{X^{*}}^{\rm cpx}$ \cite{hosCICY, HLY} of the moduli space $\mathcal{M}_{X^{*}}^{\rm cpx}$, and the geometry $\Delta_{r}^{*}\subset\mathcal{M}_{X^{*}}^{\rm cpx}$ is characterized by the corresponding normal crossing boundary divisors at the origin $o\in\Delta_{r}=\mathbb{C}^{r}$.

The following properties can be observed for an abundance of examples of complete intersection Calabi--Yau manifolds:

\begin{Observation}\label{Observation1} 
Assume $X$ and $X^{*}$ are Calabi--Yau threefolds which are mirror to each other. If Calabi--Yau threefold $X=:X_{1}$ has birational models $X_{i}$, $i=2,\dots ,s$, then there appear the corresponding boundary points $o=:o_{1}$ and $o_{i}$, $i=2,\dots ,s$, given by normal crossing divisors in~$\overline{\mathcal{M}}_{X^{*}}^{\rm cpx}$ such that
\begin{enumerate}\itemsep=0pt
\item[(1)] $o_{i}$ are LCSLs, and
\item[(2)] the A-structures of $X_{i}$ are isomorphic to the B-structures
arising from $o_{i}$.
\end{enumerate}
\end{Observation}

\begin{Observation}\label{Observation2}
Let $X$ and $X^{*}$ be as above. Corresponding to the birational map $\varphi_{ji}\colon X_{i}\dashrightarrow X_{j}$, there is a path connecting $o_{i}$ to $o_{j}$ and the connection matrix $M_{ji}$ of the B-structures such that

\begin{enumerate}\itemsep=0pt
\item[(1)] it preserves the monodromy weight filtrations, and
\item[(2)] it is integral and also compatible with the symplectic structures at each~$o_{i}$, i.e., $^{t}M_{ji}\Sigma_{j}M_{ji}$ $=\Sigma_{i}$ for the symplectic matrices $\Sigma_{i}$ representing the symplectic forms on $H^{3}(X_{b_{o_{i}}}^{*},\mathbb{Z})$.
\end{enumerate}
\end{Observation}

\begin{Remark}Recently Calabi--Yau manifolds which are derived equivalent but not birational each other have been attracting attention (see, e.g., \cite{BoCa,HTjag,HoTaCIV,Kuz2} and references therein). These are called Fourier--Mukai partners after the original work by Mukai for K3 surfaces~\cite{Mukai}. As shown in examples \cite{BoCa,HTjag,vSt}, if a Calabi--Yau threefold~$X$ has such Fourier--Mukai partners, then corresponding boundary points exist in $\overline{\mathcal{M}}_{X^{*}}^{\rm cpx}$ with the property~(2) in Observation~\ref{Observation2}, but losing the property~(1). If we have both the property~(1) and~(2), then we can see that the so-called prepotential for quantum cohomology is invariant up to quadratic terms under analytic continuations (see Proposition~\ref{prop:prepot-F} below), and hence the quantum cohomologies of birational Calabi--Yau threefolds are essentially the same (see~\cite{LLWang} for example). However, as we see in \cite{BCKvS,HoKo,HTjag,Ro}, quantum cohomologies of Fourier--Mukai partners are quite different to each other.
\end{Remark}

In this paper we will focus on Calabi--Yau threefolds given by complete intersections in toric varieties. Showing two examples which exhibit interesting birational geometry, we will make Observation~\ref{Observation2} more explicit, e.g., we will give precise descriptions about the path connecting the boundary points. Also finding some monodromy relations, we will come to the following observation:

\begin{Main} Assume a Calabi--Yau threefold~$X$ and its mirror manifold $X^{*}$ have the properties described in Observation~\ref{Observation2}. Then there are natural choices of path connecting $o_{i}$ and $o_{j}$ such that the monodromy nilpotent cones defined for each $o_{i}$ in $\overline{\mathcal{M}}_{X^{*}}^{\rm cpx}$ are glued together. We identify the resulting structure as the mirror counter part of the movable cone obtained by gluing K\"ahler cones by birational maps.
\end{Main}

The gluing will be achieved by finding monodromy relations coming from boundary divisors which have multiple tangency with some component of the discriminant (see Section~\ref{sec:Gluing-monod-I}). When writing the monodromy relations, we find a certain monodromy action of a distinguished form, which we call ``\textit{Picard--Lefschetz formula of flopping curves}'' based on the mirror correspondence (cf.\ the same forms are known in physics literatures, \cite{A-Scheideg, CandelasTwoI} for example, as strong coupling limits associated to certain contractions of curves).

\section[Complete intersection Calabi--Yau spaces from Gorenstein cones]{Complete intersection Calabi--Yau spaces\\ from Gorenstein cones}\label{sec: CICY-I}

In this section, we describe mirror symmetry of a Calabi--Yau complete intersection of the form
\begin{gather}
 X:=\left(\begin{matrix}\mathbb{P}^{4}|11111\\
\mathbb{P}^{4}|11111
\end{matrix}\right)^{2,52},\label{eq:XP4P4}
\end{gather}
i.e., a complete intersection of five general $(1,1)$ divisors in $\mathbb{P}^{4}\times\mathbb{P}^{4}$ which has Hodge numbers $\big(h^{1,1},h^{2,1}\big)=(2,52)$. In this section, we will study the A-structure of $X$.

\subsection{Cones for complete intersections and Calabi--Yau manifolds}

To describe the complete intersection $X$, let us note that we can write $X=s^{-1}(0)$ with a generic choice of a section of the bundle $\mathcal{O}(-1,-1)^{\oplus5}\to\mathbb{P}^{4}\times\mathbb{P}^{4}$. We describe this starting with the affine cone over the generalized Segre embedding $s_{1,1,1}\big(\mathbb{P}^{4}\times\mathbb{P}^{4}\times\mathbb{P}^{4}\big)$, which we write by
\begin{gather*}
U_{0}:=\operatorname{Spec}\mathbb{C}[\lambda_{i}z_{j}w_{k}\,|\,1\leq i,j,k\leq5]
\end{gather*}
with the homogeneous coordinates $\lambda_{i}$, $z_{j}$, $w_{k}$ of $\mathbb{P}^{4}$'s. Let $U\to U_{0}$ be the blow-up of the cone at the origin. It is easy to see that the exceptional divisor $E$ is isomorphic to $\mathbb{P}^{4}\times\mathbb{P}^{4}\times\mathbb{P}^{4}$. In fact, $U$ is isomorphic to the total space of the line bundle $\mathcal{O}(-1,-1,-1)\to\mathbb{P}^{4}\times\mathbb{P}^{4}\times\mathbb{P}^{4}$. Contracting one of the $\mathbb{P}^{4}$'s ($m$-th factor of $\mP_{\lambda}^{4}\times\mP_{z}^{4}\times\mP_{w}^{4}$), we have three possible contractions of $U$ which fit in the following diagram:
\begin{gather*}
\xymatrix{
& \ar[dl]_{\pi_1} U_1
\ar@{<-->}[d] \\
\;\;\;\;U_0 \;\;\;\;
& \ar[l]_{\;\;\;\;\pi_2} \;\; U_2 \;\;
\ar@{<-->}[d] &
\ar[ul] \ar[l] \ar[dl] \;\;\;\;U. \;\;\;\; \\
 & \ar[ul]^{\pi_3} U_3 }
\end{gather*}
Again, it is easy to see that $U_{\alpha}\to U_{0}$, $\alpha=1,2,3$, are small resolutions, and the geometries of~$U_{\alpha}$ are of the form $\mathcal{O}(-1,-1)^{\oplus5}\to\mathbb{P}^{4}\times\mathbb{P}^{4}$ that are birational to each other. It is worthwhile noting that if we start with the cone over $s_{1,1}\big(\mathbb{P}^{1}\times\mathbb{P}^{1}\big)$ in the above construction, the resulting geometry is the standard Atiyah flop for the small resolutions of the form $\mathcal{O}(-1)\oplus\mathcal{O}(-1)\to\mathbb{P}^{1}$.
\begin{Definition}
Consider the potential function on $U_{0}$,
\begin{gather*}
W=\sum_{i,j,k}a_{ijk}\lambda_{i}z_{j}w_{k}
\end{gather*}
with $a_{ijk}\in\mathbb{C}$ being chosen generically. Let $W_{\alpha}:=\pi_{\alpha}^{\#}W$ be the potential functions on $U_{\alpha}$. We denote the critical locus of~$W_{\alpha}$ in each~$U_{\alpha}$ by
\begin{gather*}
X_{\alpha}:=\operatorname{Crit}(W_{\alpha},U_{\alpha}), \qquad \alpha=1,2,3.
\end{gather*}
\end{Definition}

\begin{Proposition}The critical locus $X_{\alpha}$ is a Calabi--Yau complete intersection of the form \eqref{eq:XP4P4}.
\end{Proposition}
\begin{proof} By symmetry, we only consider the case $X_{1}$. To write the conditions for the criticality, it is helpful to use the homogeneous coordinate for the small resolution $U_{1}$, which is the total space $\sO(-1,-1)^{\oplus5}\to\mP_{z}^{4}\times\mP_{w}^{4}$. Let $z_{i},w_{j}$ denote the homogeneous coordinates of $\mP_{z}^{4}\times\mP_{w}^{4}$ and $\lambda_{i}$ be the fiber coordinate. Then the potential function is simply given by $W_{1}=\sum_{i,j,k}a_{ijk}\lambda_{i}z_{j}w_{k}$, which gives the conditions for the criticality $\frac{\partial W_{1}}{\partial\lambda_{i}}=\frac{\partial W_{1}}{\partial z_{j}}=\frac{\partial W_{1}}{\partial w_{k}}=0$. If we denote $\frac{\partial W_{1}}{\partial\lambda_{i}}=\sum_{j,k}a_{ijk}z_{j}w_{k}=:f_{i}(z,w)$, the conditions $\frac{\partial W_{1}}{\partial z_{j}}=\frac{\partial W_{1}}{\partial w_{k}}=0$ may be arranged into a matrix form
\begin{gather*}
\left(\begin{matrix}\nabla_{z}f_{1} & \cdots & \nabla_{z}f_{5}\\
\nabla_{w}f_{1} & \cdots & \nabla_{w}f_{5}
\end{matrix}\right)\left(\begin{smallmatrix}\lambda_{1}\\
\vdots\\
\lambda_{5}
\end{smallmatrix}\right)=0.
\end{gather*}
The last equation gives the zero section $\left\{ \lambda_{1}=\cdots=\lambda_{5}=0\right\} \simeq\mP_{z}^{4}\times\mP_{w}^{4}$
and the conditions $f_{1}(z,w)=\dots =f_{5}(z,w)=0$ give a smooth complete intersection in the zero section if we choose $a_{ijk}$ sufficiently
general. \end{proof}

\begin{Proposition}[\cite{HTcmp, HTjag}] \label{proposition3.3} $X_{\alpha}$ and $X_{\beta}$, $\alpha\not=\beta$, are birational. The birational maps $\varphi_{\beta\alpha}\colon X_{\alpha}\dashrightarrow X_{\beta}$ are given by the Atiyah flops associated to the contractions of $50$ $\mathbb{P}^{1}$s, which we summarize in the following diagram:
\begin{gather}\begin{split}&
\begin{xy}
(-26,0)*++{X_1}="Xi",
(0,0)*++{X_2}="Xii",
(26,0)*++{X_3}="Xiii",
(52,0)*++{X_1,}="XiR",
(-13,-13)*+{Z_2}="Zii",
(13,-13)*+{Z_3}="Ziii",
(39,-13)*+{Z_1}="ZiR",
\ar_{\pi_{21}} "Xi";"Zii"
\ar^{\pi_{22}} "Xii";"Zii"
\ar_{\pi_{32}} "Xii";"Ziii"
\ar^{\pi_{33}} "Xiii";"Ziii"
\ar_{\pi_{13}} "Xiii";"ZiR"
\ar^{\pi_{11}} "XiR";"ZiR"
\ar@{<-->} "Xi";"Xii"
\ar@{<-->} "Xii";"Xiii"
\ar@{<-->} "Xiii";"XiR"
\end{xy}\end{split}
\label{eq:Birat-diag-1}
\end{gather}
where $Z_{1}\subset\mathbb{P}_{z}^{4}$, $Z_{2}\subset\mathbb{P}_{w}^{4}$ and $Z_{3}\subset\mathbb{P}_{\lambda}^{4}$ are determinantal quintics defined by the $5\times5$ matrices $\big(\sum z_{j}a_{ijk}\big)$, $\big(\sum w_{k}a_{ijk}\big)$ and $\big(\sum\lambda_{i}a_{ijk}\big)$, respectively.
\end{Proposition}
We refer the references \cite{HTcmp, HTjag} for the proof of the above proposition.
\begin{Remark} \label{rem:Bat-Nil} In the above proposition, we naturally come to birational Calabi--Yau complete intersections. Some remarks related to this are in order:
\begin{enumerate}\itemsep=0pt
\item $U_{\alpha}$'s are birational to each other since they are all toric varieties with the same algebraic torus contained as a dense subset. In fact, they all have the form $\mathcal{O}(-1,-1)^{\oplus5}\to\mathbb{P}^{4}\times\mathbb{P}^{4}$. However, when defining $X_{\alpha}$ as the critical locus of the potential function, the zero section of $\mathcal{O}(-1,-1)^{\oplus5}\to\mathbb{P}^{4}\times\mathbb{P}^{4}$ is specified by the criticality condition. Hence, that $U_{\alpha}$'s are birational does not imply that $X_{\alpha}$'s are birational. The fact that $X_{\alpha}$'s are birational comes from different reasons as described in the above proposition.

\item The affine cone construction here is an example of more general method in toric geometry due to Batyrev and Borisov~\cite{BatBo}. There, the affine cone is replaced by the so-called Gorenstein cones, and actually a pair of reflexive Gorenstein cone $(C_{\nabla},C_{\Delta})$ to describe mirror symmetry. The birational geometry we observed in the above proposition has been described by the property of so-called nef-partitions of $\nabla$ by Batyrev and Nil~\cite{BatNil}. They have found that the two different (but isomorphic) nef-partitions
\begin{gather*}
\nabla=\nabla_{1}+\nabla_{2}+\cdots+\nabla_{s}=\nabla_{1}'+\nabla_{2}'+\cdots+\nabla_{s}'
\end{gather*}
sometimes results in dual nef-partitions
\begin{gather*}
\Delta=\Delta_{1}+\Delta_{2}+\cdots+\Delta_{s},\qquad \Delta'=\Delta_{1}'+\Delta_{2}'+\cdots+\Delta_{k}'
\end{gather*}
with $\Delta$ and $\Delta'$ having completely different shapes to each other. We can describe our birational Calabi--Yau threefolds in this general setting. See references \cite{BoLi, FavKel} for recent works which shed light on this general phenomenon from the derived categories of Calabi--Yau threefolds.
\end{enumerate}
\end{Remark}

\subsection[Movable cone of $X:=X_{1}$]{Movable cone of $\boldsymbol{X:=X_{1}}$} \label{sub:Movable-cone}

Let us note that the K\"ahler cone of $X(=X_{1})$ is given by $\mathcal{K}_{X}=\mathbb{R}_{>0}H_{1}+\mathbb{R}_{>0}H_{2}$ with the pull-backs $H_{1}=\pi_{11}^{*}H_{Z_{1}}$ and $H_{2}=\pi_{21}^{*}H_{Z_{2}}$ of the hyperplane classes $H_{Z_{i}}$ of $Z_{i}$, where $\pi_{ji}\colon X_{i}\to Z_{j}$ is the projection in the diagram~(\ref{eq:Birat-diag-1}).
\begin{Lemma}\label{lem:KahlerX2}\quad
\begin{enumerate}\itemsep=0pt

\item[$(1)$] Let $\mathcal{K}_{X_{2}}=\mathbb{R}_{>0}L_{Z_{2}}+\mathbb{R}_{>0}L_{Z_{3}}$
be the $K$\"ahler cone with the generators $L_{Z_{2}}=\pi_{22}^{*}H_{Z_{2}}$
and $L_{Z_{3}}=\pi_{32}^{*}H_{Z_{3}}$. By the birational map $\varphi_{21}\colon X_{1}\dashrightarrow X_{2}$,
the K\"ahler cone is transformed to
\begin{gather*}
\varphi_{21}^{*}(\mathcal{K}_{X_{2}})=\mR_{>0}H_{2}+\mathbb{R}_{>0}(4H_{2}-H_{1}).
\end{gather*}
\item[$(2)$] Similarly, let $\mathcal{K}_{X_{3}}=\mathbb{R}_{>0}M_{Z_{3}}+\mathbb{R}_{>0}M_{Z_{1}}$
be the K\"ahler cone of $X_{3}$ generated by $M_{Z_{3}}=\pi_{33}^{*}H_{Z_{3}}$
and $M_{Z_{1}}=\pi_{13}^{*}H_{Z_{1}}$, then we have
\begin{gather*}
\varphi_{31}^{*}(\mathcal{K}_{X_{3}})=\mathbb{R}_{>0}(4H_{1}-H_{2})+\mR_{>0}H_{1}
\end{gather*}
for the birational map $\varphi_{31}\colon X_{1}\dashrightarrow X_{3}$.
\end{enumerate}
\end{Lemma}
\begin{proof}
See Appendix~\ref{appendixA}. \end{proof}

\begin{Lemma}\label{lem:KahlerX3}With the divisors $L_{Z_{2}}$, $L_{Z_{3}}$ and $M_{Z_{3}}$, $M_{Z_{1}}$ defined as above, we have
\begin{gather*}
\varphi_{32}^{*}(M_{Z_{1}})=4L_{Z_{3}}-L_{Z_{2}},\qquad \varphi_{32}^{*}(M_{Z_{3}})=L_{Z_{3}}
\end{gather*}
for the birational map $\varphi_{32}\colon X_{2}\dashrightarrow X_{3}$.\end{Lemma}
\begin{proof} The second relation holds by definition. For the first relation, see Appendix~\ref{appendixA}.
\end{proof}

Now, we define the following composite of the birational maps:
\begin{gather*}
\rho:=\varphi_{13}\circ\varphi_{32}\circ\varphi_{21}
\end{gather*}
with the convention $\varphi_{ij}=\varphi_{ji}^{-1}\colon X_{j}\dashrightarrow X_{i}$ (see the diagram (\ref{eq:Birat-diag-1})).
\begin{Lemma} The birational map $\rho$ is not an automorphism of~$X$. It is of infinite order.\end{Lemma}
\begin{proof}
We show that
\begin{gather}
\rho^{*}H_{1}=-4H_{1}+15H_{2},\qquad \rho^{*}H_{2}=-15H_{1}+56H_{2} \label{eq:rho-actionH}
\end{gather}
for $\rho^{*}=\varphi_{21}^{*}\circ\varphi_{32}^{*}\circ\varphi_{13}^{*}$. Since $\varphi_{13}^{*}=\big(\varphi_{31}^{-1}\big)^{*}=(\varphi_{31})_{*}$ and using the relations $\varphi_{31}^{*}(M_{Z_{3}})=4H_{1}-H_{2},\varphi_{31}^{*}(M_{Z_{1}})=H_{1}$ in Lemma~\ref{lem:KahlerX2}(2), we have
\begin{gather*}
M_{Z_{3}}=4M_{Z_{1}}-\varphi_{13}^{*}(H_{2}),\qquad M_{Z_{1}}=\varphi_{13}^{*}(H_{1}).
\end{gather*}
Then, using Lemmas~\ref{lem:KahlerX2} and~\ref{lem:KahlerX3}, it is straightforward to evaluate $\rho^{*}(H_{i})$, e.g., $\rho^{*}(H_{1})=\varphi_{21}^{*}\circ\varphi_{32}^{*}(M_{Z_{1}})=\varphi_{21}^{*}(4L_{Z_{3}}-L_{Z_{2}})=4(4H_{2}-H_{1})-H_{2}$. From these actions of $\rho^{*}$, we see that $\rho^{*}(\sK_{X})\not=\sK_{X}$ and hence $\rho\notin\operatorname{Aut}(X).$ Also, expressing the linear
action (\ref{eq:rho-actionH}) by a matrix $\left(\begin{smallmatrix}-4 & -15\\
15 & 56 \end{smallmatrix}\right)$, we see that $\rho$ has an infinite order. \end{proof}

\begin{Proposition}\label{prop:Birat(Xi)}Suppose $X_{i}\not\simeq X_{j}$, $i\not=j$, then the groups of birational maps of $X_{i}$ are given by
\begin{gather*}
\operatorname{Bir}(X_{i})=\operatorname{Aut}(X_{i})\cdot\langle\varphi_{i1}\circ\rho\circ\varphi_{1i}\rangle.
\end{gather*}
\end{Proposition}
\begin{proof} Since arguments are similar to \cite[Lemma~6.4]{Oguiso1}, here we only give a rough sketch. Also, we only describe the case $i=1$, $\varphi_{11}=\mathrm{id}_{X}$. Take a birational map $\tau\colon X\dashrightarrow X$. We denote by~$\mathtt{E}(\tau$) the locus where $\tau$ is not defined or non-isomorphic. Consider an ample divisor $D$ and its transform $D'=(\tau^{-1})_{*}D$. Under this setting, we consider the two cases: (i) If~$D'$ is nef, then using \cite[Lemma~4.4]{Kollar} we have $D|_{\mathtt{E}(\tau^{-1})}\equiv0$, i.e., numerically equivalent to zero. Since $D$ is ample, this implies $\mathtt{E}\big(\tau^{-1}\big)=\varnothing$, i.e., $\tau\in\operatorname{Aut}(X)$. (ii)~If $D'$ is not nef, the restriction $D'|_{\mathtt{E}(\tau)}$ is not nef, too. This is because if $D'|_{\mathtt{E}(\tau)}$ were nef, then $D'=(\tau^{-1})_{*}D$ must be nef because $D$ is ample. Therefore $D'|_{\mathtt{E}(\tau)}$ is not nef and there exists a~curve $C\subset\mathtt{E}(\tau)$ such that $D'\cdot C<0$. Now, since $K_{X}\vert_{\mathtt{E}(\tau)}\equiv0$, we know that $K_{X}+\varepsilon D'$, $0<\varepsilon\ll1$, is not nef and $(X,\varepsilon D')$ is klt. From the theory of minimal models, we know that there exists an extremal ray of~$\overline{\mathrm{NE}}(X)$ and its associated contraction, which must be either $X\to Z_{1}$ or $X\to Z_{2}$ up to automorphisms. Now, corresponding to these two possibilities, we make the following diagrams:
\begin{gather}\begin{split}&
\begin{xy}
(-50,0)*++{X}="Xl",
(-30,0)*++{X_2}="Xii",
(-10,0)*++{X}="Xll",
(-40,-13)*++{Z_1}="Zi",
(10,0)*++{X}="Xr",
(30, 0)*++{X_3}="Xiii",
(50,0)*++{X.}="Xll",
(20,-13)*++{Z_2}="Zii",
(0,-5)*{\text{or}}
\ar@/^3mm/@{-->}^\tau (-46,3);(-13,2)
\ar@{-->}^{\varphi_{21}} "Xl";"Xii"
\ar "Xl";"Zi"
\ar "Xii";"Zi"
\ar@/^3mm/@{-->}^\tau (14,3);(47,2)
\ar@{-->}^{\varphi_{31}} "Xr";"Xiii"
\ar "Xr";"Zii"
\ar "Xiii";"Zii"
\end{xy}\end{split}\label{diag:XZXZ}
\end{gather}
Depending on the two cases, we set $D''=\big(\varphi_{21}\tau^{-1}\big)_{*}D$ or $D''=\big(\varphi_{31}\tau^{-1}\big)_{*}D$ and consider inductively the above two cases~(i) and~(ii) again. Due to \cite[Theorem~3.5]{Kollar}, this process terminates arriving at the case~(i) in the end. We can deduce that there are only two possibilities under the assumption $X\not\simeq X_{i}$, $i=2,3$:
\begin{gather*}
\begin{xy}
(0,0)*+{X=X_{1}\overset{_{\varphi_{21}}}{\dashrightarrow}
X_{2}\overset{_{\varphi_{32}}}{\dashrightarrow}
X_{3}\overset{_{\varphi_{13}}}{\dashrightarrow}
X_{1}\overset{_{\varphi_{21}}}{\dashrightarrow}
X_{2}\cdots\overset{_{\varphi_{13}}}{\dashrightarrow}
X_{1}\xrightarrow[^{\sim}]{\varphi_{L}}
X_{1}=X,},
(0,-12)*+{X=X_{1}\overset{_{\varphi_{31}}}{\dashrightarrow}
X_{3}\overset{_{\varphi_{23}}}{\dashrightarrow}
X_{2}\overset{_{\varphi_{12}}}{\dashrightarrow}
X_{1}\overset{_{\varphi_{31}}}{\dashrightarrow}
X_{3}\cdots\overset{_{\varphi_{12}}}{\dashrightarrow}
X_{1}\xrightarrow[^{\sim}]{\varphi_{R}}
X_{1}=X.},
\ar@/_2.5mm/@{-->}^\rho (-35,-3);(-7,-3)
\ar@/_2.5mm/@{-->}^{\rho^{-1}} (-35,-15);(-7,-15)
\end{xy}
\end{gather*}
Corresponding to these two, we have the decomposition $\tau=\varphi_{L}\rho^{n}$ or $\tau=\varphi_{R}(\rho^{-1})^{m}$ with $\varphi_{L,R}\in\operatorname{Aut}(X)$.\end{proof}

\begin{Remark} We use the assumption $X_{i}\not\simeq X_{j}$ at the very end of the above proof. If $X_{1}\simeq X_{i}$, then it is easy to deduce that we only have to include $\varphi_{i1}$ in the generators of $\operatorname{Bir}(X_{1})$. Similar modification in $\operatorname{Bir}(X_{i})$ is required if $X_{i}\simeq X_{j}$, $i\not=j$. These do not affect the form of the movable cone determined below. The assumption in the above proposition has been made just for simplicity.
\end{Remark}

Let us denote by $\operatorname{Mov}(X_{i})$ be the movable cones generated
by movable divisors on $X_{i}$. Since the transforms of movable divisors
by flops are movable, we have
\begin{gather*}
\operatorname{Mov}(X)=\operatorname{Mov}(X_{1})=\varphi_{21}^{*}\operatorname{Mov}(X_{2})=\varphi_{31}\operatorname{Mov}(X_{3}).
\end{gather*}
 The following result is known by \cite[Lemma~1]{Fry}. For completeness of our arguments, we present it here with a general proof.
\begin{Proposition}
The closure of the movable cone $\operatorname{Mov}(X)$ is given by
\begin{gather}
\overline{\operatorname{Mov}}(X)=\mathbb{R}_{\geq0}\big({-}H_{1}+(2+\sqrt{3})H_{2}\big) +\mathbb{R}_{\geq0}\big(H_{1}+(-2+\sqrt{3})H_{2}\big).\label{eq:MovX}
\end{gather}
\end{Proposition}
\begin{proof}By Lemmas \ref{lem:KahlerX2} and \ref{lem:KahlerX3}, it is easy to see that the closure of the set $\varphi_{21}^{*}(\mathcal{K}_{X_{2}})\cup\mathcal{K}_{X_{1}}\cup\varphi_{31}^{*}(\mathcal{K}_{X_{3}})$ is given by
\begin{gather*}
\overline{C}_{123}:=\mathbb{R}_{\geq0}(4H_{2}-H_{1})+\mathbb{R}_{\geq0}(4H_{1}-H_{2}).
\end{gather*}
We define
\begin{gather*}
M:=\bigcup_{n\in\mathbb{Z}}(\rho^{*})^{n}\overline{C}_{123}:=\langle\rho^{*}\rangle\cdot\overline{C}_{123}.
\end{gather*}
Then, from a linear algebra, it is straightforward to see that the r.h.s.\ of~(\ref{eq:MovX}) coincides with the closure~$\overline{M}$. Since any automorphism of $X_{i}$ preserves the generators of $\sK_{X_{i}}$ or exchanges them, using Proposition~\ref{prop:Birat(Xi)}, we have $\cup_{i}\varphi_{i1}^{*}(\operatorname{Bir}(X_{i})^{*}\sK_{X_{i}})=M$. Hence we have $\overline{M}\subset\overline{\operatorname{Mov}}(X)$.

To show the other inclusion, take a rational point $d\in\operatorname{Mov}(X)$. There exist $m\gg1$ and an effective movable divisor~$D$ such that
$md=[D]$. If $D$ is nef, then $d\in\overline{\sK}_{X}$ and hence $d\in M$. If $D$ is not nef, we do the same inductive process as in the proof of Proposition~\ref{prop:Birat(Xi)} and find a~birational map $\tau\colon X\dashrightarrow X_{i}$, $\tau\in\langle\rho,\varphi_{21}, \varphi_{31}\rangle$, such that $D'=\tau_{*}D$ is a~nef divisor on $X_{i}$, i.e., $D'\in\overline{\sK}_{X_{i}}$. Namely, we have $D=\tau^{*}D'\in\overline{\mathcal{K}}_{X}$ and $\overline{\mathcal{K}}_{X}\subset M$, which imply $\operatorname{Mov}(X_{i})(\mQ)\subset M$. Hence we have $\overline{\operatorname{Mov}}(X)\subset\overline{M}$.
\end{proof}

\begin{figure}[t]\centering
\includegraphics[width=110mm]{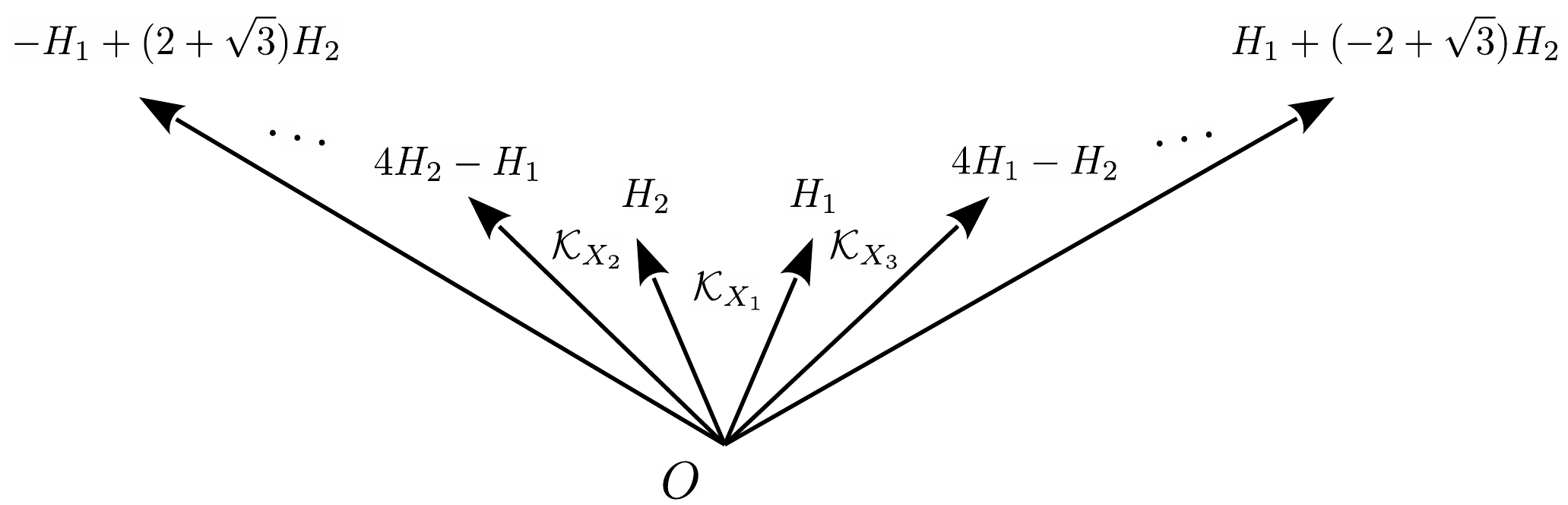}

\caption{Movable cone $\overline{\operatorname{Mov}}(X)$ in $H^2(X,\mathbb{R})$. The rays accumulate to the boundary rays of slopes $-2-\sqrt{3}$ and $-2+\sqrt{3}$.}\label{Fig1}
\end{figure}

\subsection[Mirror symmetry of $X$]{Mirror symmetry of $\boldsymbol{X}$} \label{sub:MirrorFamily}

For the complete intersection Calabi--Yau threefolds $X$ $(=X_{1})$, the mirror family can be obtained by a straightforward application of the Batyrev--Borisov toric mirror construction. However, the construction involves complications in combinatorics for toric geometry. In our case, we can avoid these complications and find the mirror family of~$X$ by the so-called orbifold mirror construction starting with a special family~\cite{HTcmp}.

Define the following special family of $X_{1}$:
\begin{gather*}
X_{\rm sp}:= \{ z_{i}w_{i}+a z_{i}w_{i+1}+b z_{i+1}w_{i}=0, \, i=1,\dots,5 \} \subset\mathbb{P}_{z}^{4}\times\mathbb{P}_{w}^{4},
\end{gather*}
where the indices of $z_{i}$, $w_{j}$ should be considered modulo $5$.

\begin{Proposition}
For general values of $a$, $b$, we have the following properties:
\begin{enumerate}\itemsep=0pt

\item[$(1)$] $X_{\rm sp}$ is singular along $20$ lines of singularity of~$A_{1}$ type.

\item[$(2)$] There exists a crepant resolution $X^{*}\to X_{\rm sp}$ with $X^{*}$ being a Calabi--Yau threefold with $h^{1,1}(X^{*})=52$, $h^{2,1}(X^{*})=2$.

\item[$(3)$] The resolution $X^{*}$ parametrized by $(a,b)\in\mathbb{C}^{2}$ defines a family $\mathfrak{X}^{*}\to\overline{\mathcal{M}}_{X^{*}}^{\rm cpx}\setminus {\rm Dis}$ with $\overline{\mathcal{M}}_{X^{*}}^{\rm cpx}=\mathbb{P}^{2}$ and ${\rm Dis}=D_{1}\cup D_{2}\cup D_{3}\cup {\rm Dis}_{0}$ where $D_{i}$ are the coordinate lines of~$\mathbb{P}^{2}$ and~${\rm Dis}_{0}$ is an irreducible $($singular$)$ curve of degree~$5$. The fiber over $\big[a^{5},b^{5},1\big]\not\in {\rm Dis}$ is given by the resolution $X^{*}$ with~$(a,b)$.

\end{enumerate}\end{Proposition}
\begin{proof}Proofs of these properties are given in \cite[Theorems~5.11 and~5.17]{HTcmp}.
\end{proof}

We can verify that all the properties in Observation~\ref{Observation1} hold for the family $\mathfrak{X}^{*}\to\mathbb{P}^{2}\setminus {\rm Dis}$.
\begin{Proposition}We set
\begin{gather*}
o_{1}=D_{1}\cap D_{2}=[0,0,1],\qquad o_{2}=D_{2}\cap D_{3}=[1,0,0],\qquad o_{3}=D_{3}\cap D_{1}=[0,1,0].
\end{gather*}
All these boundary points $o_{1}$, $o_{2}$, $o_{3}$ are LCSLs whose B-structures are identified with the A-structures of the birational models $X_{1}$, $X_{2}$ and $X_{3}$, respectively. See Fig.~{\rm \ref{Fig2}} in the next section.
\end{Proposition}

The above proposition has been derived by introducing integral and symplectic structures at each $o_{i}$ and calculating the monodromies around the divisors~$D_{i}$, see \cite[Section~6.3]{HTcmp} for details. Our focus in what follows will be gluing the monodromy cones~(\ref{eq:Mon-cone-Sig}) which are defined for each boundary point~$o_{i}$.

\section{Gluing monodromy nilpotent cones I} \label{sec:Gluing-monod-I}

For the example in the preceding section, we will find a path $o_{i}\to o_{j}$ which we can identify with the birational map $\varphi_{ji}\colon X_{i}\dashrightarrow X_{j}$ as described in Observation~\ref{Observation2}. We will find that the monodromy nilpotent cones~(\ref{eq:Mon-cone-Sig}) at each boundary point are naturally glued together by the monodromy relations coming from the path. Also, in the next section (Section~\ref{section5}) we will study another interesting example, which has no other birational models other than itself but has a~birational automorphism of infinite order.

\subsection[B-structures of $X^{*}$]{B-structures of $\boldsymbol{X^{*}}$} \label{sub:B-str-Reye}

Associated to the family $\pi\colon \mathfrak{X}^{*}\to\overline{\mathcal{M}}_{X^{*}}^{\rm cpx}\setminus {\rm Dis}$, we have the local system $R^{3}\pi_{*}\mC_{\mathfrak{X}^{*}}$ which introduces the Gauss--Manin system on the moduli space, or equivalently the Picard--Fuchs differential equation for the period integrals of holomorphic three form. This Picard--Fuchs equation has been studied in our previous work \cite{HTcmp}, where we have described the B-structure for the boundary points $o_{i}$, i.e., the integral and symplectic basis for the local solutions as well as integral monodromy matrices using the central charge formula given in \cite{IIAmonod,CentralCh} which goes back to the study of GKZ system \cite{GKZ1} in the 90's (see~\cite{HKTY,HLY} for details). Here we briefly recall the integral and symplectic basis referring to \cite{HTcmp} for its explicit form, and define the monodromy nilpotent cones for each $o_{j}$ from the monodromy matrices calculated there.

\subsubsection[B-structure at $o_1$]{B-structure at $\boldsymbol{o_1}$} \label{para:B-str-1}
Let $[-x,-y,1]\in\mathbb{P}^{2}$ be the affine coordinate with the origin $o_{1}$ (where the minus signs are required to have the canonical integral and symplectic structure based on the central charge formula). The canonical, integral and symplectic structure appears from a unique power series solution $w_{0}(x,y)$ of the Picard--Fuchs differential equation around the origin~$o_{1}$. Including the logarithmic solutions, the result can be arranged as follows:
\begin{gather}
\Pi(x,y) ={\vphantom{\big(}}^{t}\big(w_{0}(x,y),w_{1}^{(1)}(x,y),
w_{2}^{(1)}(x,y),w_{2}^{(2)}(x,y),w_{1}^{(2)}(x,y),w^{(3)}(x,y)\big)\nonumber\\
\hphantom{\Pi(x,y)}{} = {\vphantom{\Big(}}^{t}
 \left(\int_{A_{0}}\Omega_{\bm{x}},\int_{A_{1}}\Omega_{\bm{x}},
 \int_{A_{2}}\Omega_{\bm{x}},\int_{B_{2}}\Omega_{\bm{x}},\int_{B_{1}}\Omega_{\bm{x}},
 \int_{B_{0}}\Omega_{\bm{x}}\right),\label{eq:PiX}
\end{gather}
where $\{ A_{0},A_{1},A_{2},B_{2},B_{1},B_{0}\} \subset H_{3}(X_{b_{o}}^{*},\mathbb{Z})$ is a symplectic basis satisfying $A_{i}\cap B_{j}=\delta_{ij}$, $A_{i}\cap A_{j}=B_{i}\cap B_{j}=0$ representing the integral and symplectic solutions of the Picard--Fuchs equation \cite[Section~6.3.1]{HTcmp}. The monodromy matrix $\mathtt{T}_{x}$ of $\Pi(x,y)$ for a small loop around $x=0$ and similarly $\mathtt{T}_{y}$ for $y=0$ have been determined as follows:
\begin{gather*}
\mathtt{T}_{x}=\left(\begin{smallmatrix}1 & 0 & 0 & 0 & 0 & 0\\
1 & 1 & 0 & 0 & 0 & 0\\
0 & 0 & 1 & 0 & 0 & 0\\
5 & 10 & 10 & 1 & 0 & 0\\
2 & 5 & 10 & 0 & 1 & 0\\
-5 & -3 & -5 & 0 & -1 & 1
\end{smallmatrix}\right),\qquad \mathtt{T}_{y}=\left(\begin{smallmatrix}1 & 0 & 0 & 0 & 0 & 0\\
1 & 1 & 0 & 0 & 0 & 0\\
1 & 0 & 1 & 0 & 0 & 0\\
2 & 10 & 5 & 1 & 0 & 0\\
5 & 10 & 10 & 0 & 1 & 0\\
-5 & -5 & -3 & -1 & 0 & 1
\end{smallmatrix}\right).
\end{gather*}
We define
\begin{gather}
\mathcal{B}_{1}:= \{ \alpha_{0},\alpha_{1},\alpha_{2},\beta_{2},\beta_{1},\beta_{0} \} \subset H^{3}(X_{b_{o}}^{*},\mathbb{Z})\label{eq:base-B1}
\end{gather}
to be the dual basis satisfying $\int_{A_{i}}\alpha^{j}=\delta_{i}^{\;j}=\int_{B_{i}}\beta^{j}$ and $\int_{A_{i}}\beta^{j}=\int_{B_{i}}\alpha^{j}=0$. Since the monodromy actions on the period integrals, i.e., on $H_{3}(X_{b_{o}}^{*},\mathbb{Z})$, are translated into the dual space via the transpose and inverse, we define the linear action $N_{\lambda}=\sum\lambda_{i}N_{i}$ on $H^{3}(X_{b_{o}}^{*},\mathbb{Z})$ by
\begin{gather*}
N_{1}:=-\,{}^{t}(\log\mathtt{T}_{x}),\qquad N_{2}:=-\,{}^{t}(\log\mathtt{T}_{y}).
\end{gather*}
Then we define the monodromy nilpotent cone at $o_{1}$ by
\begin{gather}
\Sigma_{o_{1}}:=\left\{ \sum\lambda_{i}N_{i}\,|\,\lambda_{i}>0\right\} \subset\operatorname{End}\big(H^{3}(X_{b_{0}},\mathbb{Q})\big).\label{eq:Cone1}
\end{gather}
For general values of $\lambda_{i}>0$, it is easy to see that the nilpotent matrix~$N_{\lambda}$ induces the monodromy weight filtration $W_{0}\subset W_{2}\subset W_{4}\subset W_{6}=H^{3}(X_{b_{o}}^{*},\mathbb{Q})$
given by
\begin{alignat}{3}
& W_{0}=\langle\alpha_{0}\rangle,\qquad && W_{2}=\langle\alpha_{0},\alpha_{1},\alpha_{2}\rangle,&\nonumber\\
& W_{4}=\langle\alpha_{0},\alpha_{1},\alpha_{2},\beta_{2},\beta_{1}\rangle,\qquad && W_{6}=\langle\alpha_{0},\alpha_{1},\alpha_{2},\beta_{2},\beta_{1},\beta_{0}\rangle.& \label{eq:W2i-filt}
\end{alignat}

Using the matrices $N_{1}$, $N_{2}$, it is easy to see the following property:
\begin{Proposition}\label{prop:Nijk-N0}We have
\begin{gather*}
N_{i}N_{j}N_{k}=C_{ijk}\mathtt{N}_{0}
\end{gather*}
with totally symmetric $C_{ijk}$ given by $C_{111}=C_{222}=5$, $C_{112}=C_{122}=10$ and $C_{ijk}=0$ for other cases, and $\mathtt{N}_{0}=\left(\begin{smallmatrix}0 & 1\\
O_{5} & 0
\end{smallmatrix}\right)$ where $O_{5}$ is the zero matrix of size $5\times5$.\end{Proposition}
\begin{Remark}\label{rem:T-and-dual-T} As we see above, the monodromy matrices of the period integrals act on $H_{3}(X_{b_{o}}^{*}{,}\mathbb{Z})$ while the monodromy weight filtration is defined in the dual space $H^{3}(X_{b_{o}}^{*},\mathbb{Z})$. Hence, we translate any monodromy matrix~$\mathtt{A}$ obtained from the analytic continuations of the period integ\-ral~$\Pi(x,y)$ to the corresponding matrix $A$ in the dual space by $A= {}^{t}\mathtt{A}^{-1}$.
\end{Remark}

\subsubsection[B-structures at $o_2$, $o_3$]{B-structures at $\boldsymbol{o_2}$, $\boldsymbol{o_3}$} \label{para: B-str-23} In a similar way to the last paragraph, we determine the B-structure from the boundary points~$o_{2}$ and~$o_{3}$, which are given by the origins of the affine charts $[1,-y',-x']\in\mathbb{P}^{2}$ and \smash{$[-x'',1,-y'']\in\mathbb{P}^{2}$}. As described in detail in \cite[Section~6.3.1]{HTcmp}, we have the canonical integral and symplectic basis
\begin{gather}
\Pi'(x',y')=x'\Pi(x',y')\qquad \text{and} \qquad \Pi''(x'',y'')=y''\Pi(x'',y'')\label{eq:PiXpXpp}
\end{gather}
in terms of the same $\Pi(x,y)$ as (\ref{eq:PiX}) for $o_{2}$ and $o_{3}$, respectively. Since both of (\ref{eq:PiXpXpp}) have essentially the same form as~$\Pi(x,y)$, we have
\begin{gather}
\mathtt{T}_{x'}^{\prime}=\mathtt{T}_{x''}^{\prime\prime}=\mathtt{T}_{x}\qquad \text{and} \qquad \mathtt{T}_{y'}^{\prime}=\mathtt{T}_{y''}^{\prime\prime}=\mathtt{T}_{y}\label{eq:TxTyForAll}
\end{gather}
for the monodromy matrices with the base points $b_{o}'$ and $b_{o}''$ near the origins. Hence for~$o_{2}$ and~$o_{3}$ we have isomorphic
B-structures with
\begin{gather*}
\tilde{N}_{1}'=\log T_{x'}',\qquad \tilde{N}_{2}'=\log T_{y'}'\qquad \text{and} \qquad \tilde{N}_{1}''=\log T_{x''}'',\qquad \tilde{N}_{2}''=\log T_{y''}'',
\end{gather*}
where $T_{x'}'=(\,^{t}\mathtt{T}_{x'}')^{-1}$, $T_{y'}'=(\,^{t}\mathtt{T}_{y'}')^{-1}$ and similarly for $T_{x''}''$, $T_{y''}''$. These nilpotent matrices determine the respective monodromy weight filtrations in $H^{3}(X_{b_{o}'}^{*},\mathbb{Q})$ and $H^{3}(X_{b_{o}''}^{*},\mathbb{Q})$ with the basis
\begin{gather}
\left\{ \alpha_{0}',\alpha_{1}',\alpha_{2}',\beta_{2}',\beta_{1}',\beta_{0}'\right\} \qquad \text{and} \qquad \left\{ \alpha_{0}'',\alpha_{1}'',\alpha_{2}'',\beta_{2}'',\beta_{1}'',\beta_{0}''\right\} ,\label{eq:base-B23}
\end{gather}
as described above. We denote the monodromy nilpotent cones at $o_{2}$ and $o_{3}$ by
\begin{gather}
\Sigma_{o_{2}}'=\left\{ \sum\lambda_{i}\tilde{N}_{i}'\,|\,\lambda_{i}>0\right\} \subset\operatorname{End}\big(H^{3}(X_{b_{o}'}^{*},\mathbb{Q})\big),\nonumber\\
\Sigma_{o_{3}}''=\left\{ \sum\lambda_{i}\tilde{N}_{i}''\,|\,\lambda_{i}>0\right\} \subset\operatorname{End}\big(H^{3}(X_{b_{o}''}^{*},\mathbb{Q})\big).
\label{eq:Cones23}
\end{gather}
These are the B-structures which we identify with the A-structures of the birational models~$X_{2}$ and~$X_{3}$, respectively, in~\cite{HTcmp}.

\subsection{Gluing the monodromy nilpotent cones} \label{sub:Gluing-MN-subsec}

The monodromy matrices are transformed by conjugation when the base point is changed along a path. We can transform the monodromy nilpotent cones~(\ref{eq:Cones23}) into $H^{3}(X_{b_{o}}^{*},\mathbb{Q})$ once we fix paths $p_{b_{0}'\leftarrow b_{o}}$ and $p_{b_{0}''\leftarrow b_{o}}$. Let us denote by $\varphi_{b_{o}'b_{o}}$ the resulting isomorphism $\varphi_{b_{o}'b_{o}}\colon H^{3}(X_{b_{o}}^{*},\mathbb{Q})\simeq H^{3}(X_{b_{o}'}^{*},\mathbb{Q})$ and similarly for $\varphi_{b_{o}''b_{o}}$. We define the transforms of the nilpotent cones~(\ref{eq:Cones23}) by these isomorphisms by
\begin{gather*}
\Sigma_{o_{2}}:=(\varphi_{b_{o}'b_{o}})^{-1} \Sigma_{o_{2}}' \varphi_{b_{o}'b_{0}} ,\qquad \Sigma_{o_{3}}:=(\varphi_{b_{o}''b_{o}})^{-1} \Sigma_{o_{3}}'' \varphi_{b_{o}''b_{o}}.
\end{gather*}
Then the cones $\Sigma_{o_{2}}$ and $\Sigma_{o_{3}}$ are generated by
\begin{gather*}
N_{i}':=(\varphi_{b_{o}'b_{o}})^{-1} \tilde{N}_{i}' \varphi_{b_{o}'b_{0}},\qquad N_{i}'':=(\varphi_{b_{o}''b_{o}})^{-1} \tilde{N}_{i}'' \varphi_{b_{o}''b_{0}}, \qquad i=1,2,
\end{gather*} respectively. Note that $\Sigma_{o_{1}}$, $\Sigma_{o_{2}}$, $\Sigma_{o_{3}}$ are cones in $\operatorname{End}\big(H^{3}(X_{b_{o}}^{*},\mathbb{Q})\big)$.

\subsubsection[Path $p_{{o_2}\leftarrow {o_1}}$]{Path $\boldsymbol{p_{{o_2}\leftarrow {o_1}}}$}

The transform $\Sigma_{o_{2}}$ of the nilpotent cone obviously depends on the choice of the path. Looking the moduli space $\overline{\mathcal{M}}_{X^{*}}^{\rm cpx}$ closely, we find that there is a natural choice of the path by which the cone $\Sigma_{o_{2}}$ is glued with $\Sigma_{o_{1}}$ along a common face (boundary ray) of them.

The moduli space $\overline{\mathcal{M}}_{X^{*}}^{\rm cpx}$ has been studied in detail in~\cite{HTcmp}. Here we recall the structure of the discriminant ${\rm Dis}={\rm Dis}_{0}\cup D_{x}\cup D_{y}\cup D_{z}$. As we schematically reproduce the results in Fig.~\ref{Fig1}, the irreducible component ${\rm Dis}_{0}$ of the discriminant touches the divisor $D_{y}=\{ y=0\} $ at $(x,y)=(1,0)$ with fifth-order tangency as we can see in the expression
\begin{gather*}
{\rm Dis}_{0}=\big\{ (1-x-y)^{5}-5^{4}xy(1-x-y)^{2}+5^{5}xy(xy-x-y)=0\big\} .
\end{gather*}
We introduce the affine chart $\mathbb{C}_{(1,0)}^{2}$ with the origin $(1,0)$. After blowing-up at the origin five times, we can remove the tangential intersection of the proper transform $\widetilde{{\rm Dis}}_{0}$ of ${\rm Dis}_{0}$ with the exceptional divisors (see Fig.~\ref{Fig2}). We denote the exceptional divisors by $E_{1},\dots,E_{5}$.
\begin{Definition}
Let $q_{12}$ be a point near the intersection $E_{1}\cap D_{y}$, and $b_{o}$, $b_{o}'$ be points near the origins $o_{1}$ and $o_{2}$, respectively. We define a path $p_{b_{o}'\leftarrow b_{0}}$ to be the composite path $p_{b_{o}'\leftarrow q_{12}}\circ p_{q_{12}\leftarrow b_{o}}$ of the following straight lines:
\begin{gather*}
\begin{aligned}p_{q_{12}\leftarrow b_{o}}=\{(1-t)b_{o}+tq_{12}\,|\,0\leq t\leq1\},\\
p_{b_{o}'\leftarrow q_{12}}=\{(1-t)q_{12}+tb_{o}'\,|\,0\leq t\leq1\}.
\end{aligned}
\end{gather*}
\end{Definition}

\subsubsection[The isomorphisms $\varphi_{b_o'b_0}$, $\varphi_{b_o''b_o'}$ and $\varphi_{b_ob_o''}$]{The isomorphisms $\boldsymbol{\varphi_{b_o'b_0}}$, $\boldsymbol{\varphi_{b_o''b_o'}}$ and $\boldsymbol{\varphi_{b_ob_o''}}$}

We first calculate the connection matrix of the local solution $\Pi(x,y)$ along the path $p_{b_{o}'\leftarrow b_{o}}$.

\begin{figure}[t]\centering
\includegraphics[width=100mm]{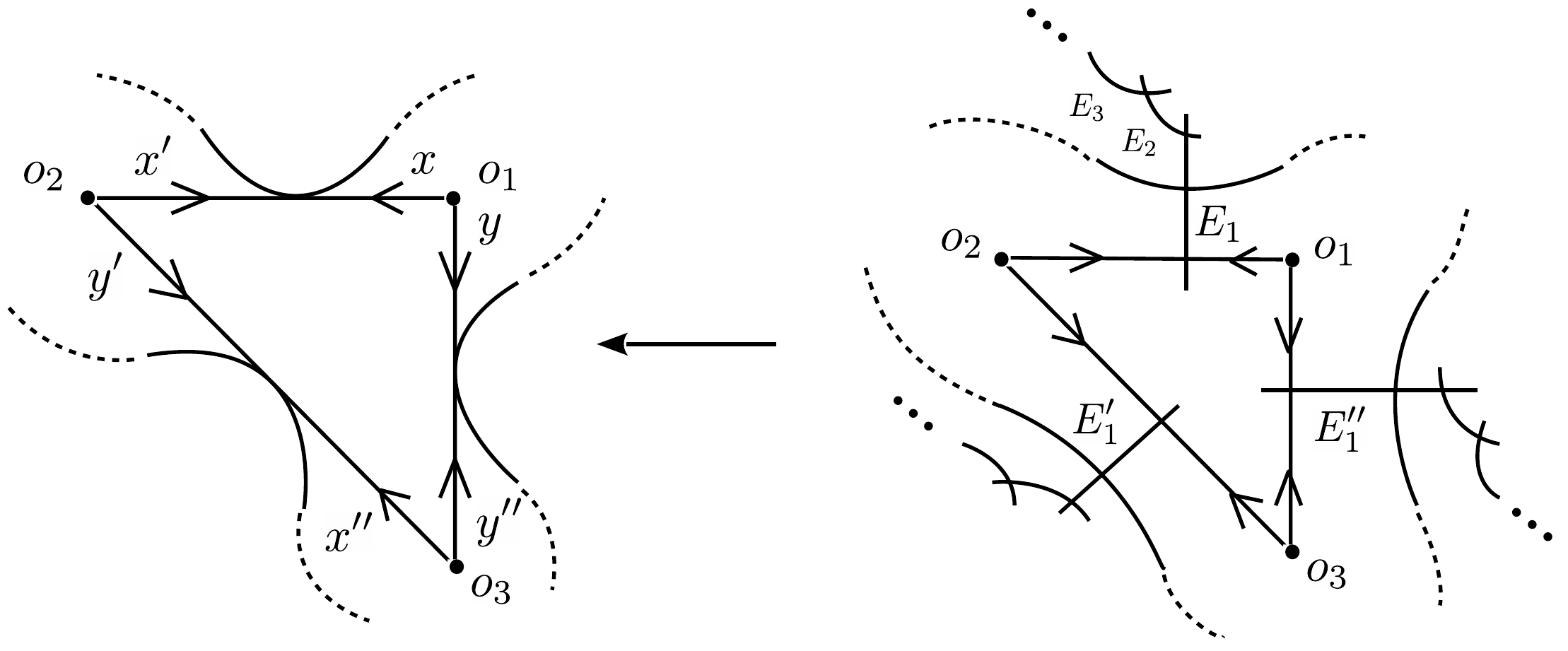}

\caption{Blowing-up the moduli space $\overline{\mathcal M}_{X^*}^{\rm cpx}={\mathbb P}^2$. To remove the tangential intersections at $[1,1,0]$, $[1,0,1]$, $[0,1,1]$, we blow-up five times at each of the three points. The exceptional divisors $E_1$, $E_1'$, $E_1''$ are normal crossing with the proper transform $\widetilde{{\rm Dis}_0}$ of the discriminant. The affine coordinates are introduced by the relations $[-x,-y,1]=[1,-y',-x']=[-x'',1,-y'']$. }\label{Fig2}
\end{figure}

\begin{Proposition}\label{prop:connect-phi-21}With respect to the basis \eqref{eq:base-B1} and \eqref{eq:base-B23}, the isomorphism $\varphi_{b_{o}'b_{o}}\colon H^{3}(X_{b_{o}}^{*},\mathbb{Q})\simeq H^{3}(X_{b_{o}'}^{*},\mathbb{Q})$ along the path $p_{b_{o}'\leftarrow b_{o}}$ is given by
\begin{gather*}
\varphi_{b_{o}'b_{o}}=\left(\begin{matrix}-1 & 0 & 0 & 0 & 0 & 0\\
0 & 1 & -4 & 2 & 25 & 0\\
0 & 0 & -1 & 0 & -2 & 0\\
0 & 0 & 0 & -1 & -4 & 0\\
0 & 0 & 0 & 0 & 1 & 0\\
0 & 0 & 0 & 0 & 0 & -1
\end{matrix}\right).
\end{gather*}
This isomorphism preserves the monodromy weight filtrations and also the symplectic structures described in Section~{\rm \ref{para: B-str-23}}. \end{Proposition}

\begin{proof} To determine the matrix form of $\varphi_{b_{o}'b_{o}}$, we do first the analytic continuation of the period integral $\Pi(x,y)$ along
the path $p_{q_{12}\leftarrow b_{o}}$ by making local solutions around $q_{12}=E_{1}\cap D_{y}$ in terms of the blow-up coordinates $s_{1}=x-1$, $s_{2}=\frac{y}{(1-x)^{5}}$ which represent $q_{12}$ by \smash{$s_{1}=s_{2}=0$}. There are two local solutions which are given by regular powerseries, and others contain logarithmic singularities given by $\log s_{1}$ and $\log s_{2},\dots,(\log s_{2})^{3}$. For a fixed value of~$y$, $|y|\ll1$, we analytically continue these solutions to $\Pi(x,y)$ as functions of $s_{1}=x-1$. Note that, under the analytic continuation, the powers of $\log y$ are unchanged. Hence the connection matrix follows from the analytic continuation of the period integrals $\Pi(x,0)$ where we set $\log y=0$ and $y=0$. In our actural calculation, we set $s_{2}=0$ and $\log s_{2}=-5\log(1-x)$ for the local solutions around $(s1,s2)=(0,0)$, and relate these solutions numerically to $\Pi(x,0)$ using powerseries expansions with sufficiently high degrees. In a similar way, we can calculate the connection matrix for the latter half $p_{b_{o}'\leftarrow q_{12}}$ of the path $p_{b_{o}'\leftarrow b_{o}}$. Actually, we can avoid the above numerical calculations finding an analytic formula for $\Pi(x,0)$. However, since the details are technical, we will report them elsewhere. It is clear that the connection matrix $\varphi_{b_{o}'b_{o}}$ preserves the filtrations since it is block diagonal with respect to the basis compatible with the filtrations $W_{0}\subset W_{2}\subset W_{4}\subset W_{6}=H^{3}(X_{b_{o}}^{*},\mathbb{Q})$ and $W_{0}'\subset W_{2}'\subset W_{4}'\subset W_{6}'=H^{3}(X_{b_{o}'}^{*},\mathbb{Q})$. Moreover, we can verify directly that it preserves the symplectic structure given by~(\ref{eq:base-B23}).
\end{proof}

From the forms of period integrals given in (\ref{eq:PiXpXpp}), it is easy to deduce that we have the isomorphisms
\begin{gather*}
\varphi_{b_{o}''b_{o}'}\colon \ H^{3}(X_{b_{o}'}^{*},\mathbb{Z})\simeq H^{3}(X_{b_{o}''}^{*},\mathbb{Z})\qquad \text{and} \qquad \varphi_{b_{o}b_{o}''}\colon \ H^{3}(X_{b_{o}''}^{*},\mathbb{Z})\simeq H^{3}(X_{b_{o}}^{*},\mathbb{Z})
\end{gather*}
by simply exchanging the bases $\alpha_{1}\leftrightarrow\alpha_{2}$ and $\beta_{1}\leftrightarrow\beta_{2}$ suitably, i.e., $\varphi_{b_{o}''b_{o}'}=\varphi_{b_{o}'b_{o}}\mathrm{p}_{23}\mathrm{p}_{45}$ and $\varphi_{b_{o}b_{o}''}=\mathrm{p}_{23}\mathrm{p}_{45}\varphi_{b_{o}'b_{o}}\mathrm{p}_{23}\mathrm{p}_{45}$ with the permutation matrices $\mathrm{p}_{ij}$ for the transposition $(i,j)$. Explicitly, they are given by
\begin{gather*}
\varphi_{b_{o}''b_{o}'}=\left(\begin{smallmatrix}-1 & 0 & 0 & 0 & 0 & 0\\
0 & -4 & 1 & 25 & 2 & 0\\
0 & -1 & 0 & -2 & 0 & 0\\
0 & 0 & 0 & -4 & -1 & 0\\
0 & 0 & 0 & 1 & 0 & 0\\
0 & 0 & 0 & 0 & 0 & -1
\end{smallmatrix}\right),\qquad \varphi_{b_{o}b_{o}''}=\left(\begin{smallmatrix}-1 & 0 & 0 & 0 & 0 & 0\\
0 & -1 & 0 & -2 & 0 & 0\\
0 & -4 & 1 & -25 & 2 & 0\\
0 & 0 & 0 & 1 & 0 & 0\\
0 & 0 & 0 & -4 & -1 & 0\\
0 & 0 & 0 & 0 & 0 & -1
\end{smallmatrix}\right).
\end{gather*}
Here we note that these isomorphisms preserve the monodromy weight filtrations and also the symplectic structures described in Sections~\ref{para:B-str-1} and~\ref{para: B-str-23}. Also it should be noted that we have verified Observation~\ref{Observation2} in Section~\ref{para:obs-path} in the present case.

Let us introduce the following notation:
\begin{gather*}
\cvarphi_{21}:=\varphi_{b_{o}'b_{o}},\qquad \cvarphi_{32}:=\varphi_{b_{o}''b_{o}'},\qquad \cvarphi_{13}:=\varphi_{b_{o}b_{o}''}
\end{gather*}
and also set $\cvarphi_{ij}:=\cvarphi_{ji}^{-1}$. As this notation indicates, we expect certain correspondence of these~$\cvarphi_{ij}$ to the birational maps $\varphi_{ij}\colon X_{j}\dashrightarrow X_{i}$ under mirror symmetry. In order to make this more explicit, we note the groupoid structure associated to the isomorphisms $\cvarphi_{ij}$.
\begin{Definition}\label{def:crho}We denote by $G_{\{1,2,3\}}$ the groupoid generated by $\cvarphi_{21}$, $\cvarphi_{32}$, $\cvarphi_{13}$.
\end{Definition}
Let $G_{ij}$ be the subset of $G_{\{1,2,3\}}$ consisting of elements $\cvarphi_{ii_{1}}\cvarphi_{i_{1}i_{2}}\cdots\cvarphi_{i_{k}j}$, $k\geq0$. It is easy to see that
\begin{gather*}
G_{11}=\big\{ \crho^{n}\,|\, n\in\mathbb{Z}\big\} ,\qquad G_{21}=\big\{ \cvarphi_{21}\crho^{n}\,|\, n\in\mathbb{Z}\big\} ,\qquad G_{31}=\big\{ \cvarphi_{31}\crho^{n}\,|\, n\in\mathbb{Z}\big\} ,
\end{gather*}
where set $\crho:=\cvarphi_{13}\cvarphi_{32}\cvarphi_{21}$.

\subsubsection{Groupoid actions on the nilpotent cones} \label{para:Sigma(n)-and-Ni(n)}
We define the following conjugates of the nilpotent cones (\ref{eq:Cone1}) and (\ref{eq:Cones23}):
\begin{gather*}
\Sigma_{o_{1}}^{(n)}:=\big(\crho^{-1}\big)^{n}\Sigma_{o_{1}}\crho^{n},\\
\Sigma_{o_{2}}^{(n)}:=\big(\crho^{-1}\big)^{n}\cvarphi_{21}^{-1}\Sigma_{o_{2}}'\cvarphi_{21}\crho^{n} =\big(\crho^{-1}\big)^{n}\Sigma_{o_{2}}\crho^{n},\\
\Sigma_{o_{3}}^{(n)}:=\big(\crho^{-1}\big)^{n}\cvarphi_{31}^{-1}\Sigma_{o_{3}}''\cvarphi_{31}\crho^{n} =\big(\crho^{-1}\big)^{n}\Sigma_{o_{3}}\crho^{n}.
\end{gather*}
These are cones in $\operatorname{End}\big(H^{3}(X_{b_{o}}^{*},\mathbb{R})\big)$ and generalize the nilpotent cones $\Sigma_{o_{k}}=\Sigma_{\sigma_{k}}^{(0)}$, $k=1,2,3$, introduced in the beginning of this subsection. It is easy to see that these cones are generated by
\begin{gather*}
N_{i}(n):=\big(\crho^{-1}\big)^{n} N_{i} \crho^{n},\qquad N_{i}'(n):=\big(\crho^{-1}\big)^{n} N_{i}' \crho^{n},\qquad N_{i}''(n):=\big(\crho^{-1}\big)^{n}\,N_{i}'' \crho^{n},
\end{gather*}
respectively, where we set $N_{i}':=\cvarphi_{21}^{-1}\tilde{N}_{i}'\cvarphi_{21}$ and $N_{i}'':=\cvarphi_{31}^{-1}\tilde{N}_{i}''\cvarphi_{31}$, $i=1,2$.

\subsubsection{Monodromy relations}
To see how the (closure of the) cone $\Sigma_{o_{2}}=\Sigma_{o_{2}}^{(0)}$ is connected to (that of) $\Sigma_{o_{1}}=\Sigma_{o_{1}}^{(0)}$, we calculate the generators $N_{i}'$ in $\operatorname{End}\big(H^{3}(X_{b_{o}}^{*},\mathbb{Z})\big)$. By the definition of~$N_{i}'$, it suffices to calculate
\begin{gather*}
T_{x'}:=\cvarphi_{21}^{-1} T_{x'}^{\prime} \cvarphi_{21},\qquad T_{y'}:=\cvarphi_{21}^{-1} T_{y'}^{\prime} \cvarphi_{21},
\end{gather*}
since we can use $T_{x'}'=T_{x}$, $T_{y'}'=T_{y}$ for the local monodromy matrices as we remarked in~(\ref{eq:TxTyForAll}). Similarly, using the connection matrix along the path $p_{q_{12}\leftarrow b_{o}}$, we can express the local monodromy around the exceptional divisor~$E_{1}$ as a linear (integral and symplectic) action on $H^{3}(X_{b_{o}}^{*},\mathbb{Z})$ which we denote by a matrix $T_{E_{1}}$ using the basis $\mathcal{B}_{1}$ in~(\ref{eq:base-B1}).
\begin{Proposition}[`Picard--Lefschetz formula' for flopping curves] \label{prop:Picard-Lefschetz-Reye}Using the basis $\mathcal{B}_{1}$ in~\eqref{eq:base-B1}, we have
\begin{gather*}
T_{E_{1}}=\left(\begin{smallmatrix}1\\
 & 1 & & & 50\\
 & & 1\\
 & & & 1\\
 & & & & 1\\
 & & & & & 1
\end{smallmatrix}\right), \qquad \text{i.e.}, \qquad \begin{cases}
\alpha_{1}\to\alpha_{1}+50\beta_{1},\\
\beta_{1}\to\beta_{1},\\
\alpha_{i}=\alpha_{i},\quad \beta_{i}=\beta_{i},\quad i\not=1.
\end{cases}
\end{gather*}
\end{Proposition}
\begin{proof}As sketched briefly in the proof of Proposition \ref{prop:connect-phi-21}, we make the local solutions of the Picard--Fuchs equation around the point of the blow-up $q_{12}=E_{1}\cap D_{y}$, and calculate the local monodromy around the divisor~$E_{1}$. The claimed monodromy
follows from the analytic continuation of the local solutions in the period integral $\Pi(x,y)$ near the origin~$o_{1}$. In our actual calculations, we only have powerseries expressions for the local solutions around $q_{12}$ and evaluate them numerically for the analytic continuation. However, as in Proposition~\ref{prop:connect-phi-21}, we can attain sufficient precision having an analytic formula for $\Pi(x,0)$. \end{proof}

\begin{Remark} \label{rem:Remark-TE1} The `Picard--Lefschetz formula' above is written using the symplectic basis $\{ \alpha_{i},\beta_{j}\} $ of $H^{3}(X_{b_{o}}^{*},\mathbb{Z})$. When we translate this into the dual basis $\{ A_{i},B_{j}\}$ of $H_{3}(X_{b_{o}}^{*},\mathbb{Z})$, we have
\begin{gather*}
A_{1}\to A_{1},\qquad B_{1}\to B_{1}-50A_{1}
\end{gather*}
with the rest of the basis left invariant. This should be contrasted to the genuine Picard--Lefschetz monodromy
\begin{gather*}
A_{0}\to A_{0}+B_{0},\qquad B_{0}\to B_{0},
\end{gather*}
which we can see for the monodromy transformation around the proper transform $\widetilde{{\rm Dis}}_{0}$ of the discriminant. In the latter case, we see the topology of the cycles as $A_{0}\approx T^{3}$, $B_{0}\approx S^{3}$, where~$S^{3}$ is a vanishing cycle and~$T^{3}$ is its dual torus cycle. Recently, the construction of the $A_{k}$-cycles $(k\not=0)$ has been discussed in general in~\cite{RuddSie}. It is interesting to see how the dual $B_{k}$-cycles are constructed, and how the above `Picard--Lefschetz formula' are explained by the geometry of these cycles.
\end{Remark}

\begin{Proposition}We have the following monodromy relations:
\begin{gather}
\boxed{T_{x'}=T_{E_{1}}^{-1}T_{x}^{-1}T_{y}^{4},\qquad T_{y'}=T_{y}.}\label{eq:T-relation}
\end{gather}
\end{Proposition}
\begin{proof}Recall that we have the relations $T_{x}=(\,^{t}\mathtt{T}_{x})^{-1}$, $T_{y}=(\,^{t}\mathtt{T}_{y})^{-1}$ (see Remark~\ref{rem:T-and-dual-T}). Then both the relations can be verified directly using the explicit forms of~$T_{x}$,~$T_{y}$ given in Section~\ref{para:B-str-1} and~$T_{x'}$, $T_{y'}$, $T_{E_{1}}$ above. The second relation also follows from the fact that the divisor \smash{$\{ y=0\} =\{y'=0\} $} intersects normally with the exceptional divisor~$E_{1}$ of the blowing-up.
\end{proof}

We have arrived at~(\ref{eq:T-relation}) by explicit monodromy calculations. It is natural to expect to have a conceptual derivation of~(\ref{eq:T-relation}) by studying mirror symmetry of conifold transitions, but we have to this to future investigations. Instead, in the rest of this section, we will interpret the monodromy relation~(\ref{eq:T-relation}).

\begin{Proposition}
\label{prop:Sigma0+Sigma2}The following properties hold:
\begin{enumerate}\itemsep=0pt
\item[{\rm (1)}] Generators $N_{i}'$ are expressed as
\begin{gather*}
N_{1}'=4N_{2}-N_{1}+\Delta_{1,0}',\qquad N_{2}'=N_{2},
\end{gather*}
where $\Delta_{1,0}'$ is a non-zero element of $\operatorname{End}\big(H^{3}(X_{b_{o}}^{*},\mathbb{R})\big)$ which annihilates the subspace $W_{2}$, i.e., $\Delta_{1,0}'\vert_{W_{2}}=0$.

\item[{\rm (2)}] The monodromy nilpotent cones $\Sigma_{o_{2}}=\mathbb{R}_{>0}N_{1}'+\mathbb{R}_{>0}N_{2}'$ and $\Sigma_{o_{1}}$ glue together along \smash{$N_{2}'=N_{2}$}. They are not in a two dimensional plane in $\operatorname{End}\big(H^{3}(X_{b_{o}}^{*},\mathbb{R})\big)$.
\end{enumerate}\end{Proposition}
\begin{proof} The properties in (1) are based on explicit calculations using~(\ref{eq:T-relation}). The second relation $N_{2}'=\log T_{y'}=\log T_{y}=N_{2}$ is clear. For the first relation, by evaluating the matrix logarithms, we have
\begin{gather*}
\Delta_{1,0}'=N_{1}'-(4N_{2}-N_{1})=\left(\begin{smallmatrix}0 & 0 & 0 & 0 & 25 & -\frac{25}{3}\\
0 & 0 & 0 & 0 & -50 & 25\\
0 & 0 & 0 & 0 & 0 & 0\\
0 & 0 & 0 & 0 & 0 & 0\\
0 & 0 & 0 & 0 & 0 & 0\\
0 & 0 & 0 & 0 & 0 & 0
\end{smallmatrix}\right).
\end{gather*}
From this triangular form, we see the claimed property of $\Delta_{1,0}'$ (see also~(\ref{eq:W2i-filt})). The claims in~(2) are clear from~(1) and also from the fact that the cone~$\Sigma_{o_{1}}$ is generated by $N_{1}=\log T_{x}$ and $N_{2}=\log T_{y}$.
\end{proof}

\begin{Remark}(1) It should be observed that, under the identification
\begin{gather*}
L_{Z_{3}}\leftrightarrow N_{1}',\qquad L_{Z_{2}}\leftrightarrow N_{2}'\qquad \text{and}\qquad H_{1}\leftrightarrow N_{1},\qquad H_{2}\leftrightarrow N_{2},
\end{gather*}
Proposition \ref{prop:Sigma0+Sigma2} above is the mirror counter part for the gluing of K\"ahler cones described in Lemma~\ref{lem:KahlerX2}.

(2) If the first monodromy relation of (\ref{eq:T-relation}) were $T_{x'}=T_{x}^{-1}T_{y}^{4}$, then we would have
\begin{gather*}
N_{1}'=4N_{2}-N_{1},\qquad N_{2}'=N_{2},
\end{gather*}
since $T_{x}$ and $T_{y}$ are commutative. These relations are exactly the same as those we have seen in Lemma~\ref{lem:KahlerX2}. However the presence of $T_{E_{1}}$ prevents this exact correspondence. We will see that $T_{E_{1}}$ represents the first order quantum correction coming from the 50 flopping curves of the contraction $X_{1}\dashrightarrow Z_{1}$. Thus the gluing relation found in Proposition~\ref{prop:Sigma0+Sigma2}(1) naturally encodes the first order quantum corrections.
\end{Remark}

\subsubsection{Gluing nilpotent cones} Before going into general descriptions, it will be helpful to see that the cone $\Sigma_{o_{1}}$ is glued with~$\Sigma_{o_{3}}$ along $N_{1}$ in a similar way as above. Let
us define
\begin{gather*}
T_{x''}:=\cvarphi_{31}^{-1}T_{x''}''\cvarphi_{31},\qquad T_{y''}:=\cvarphi_{31}^{-1}T_{y''}''\cvarphi_{31},
\end{gather*}
and also $T_{E_{1}''}$ for the monodromy matrix around the exceptional divisor $E_{1}''$. Observing the symmetry in Fig.~\ref{Fig2} and~(\ref{eq:PiXpXpp}), it is easy to deduce the following monodromy relations
\begin{gather}
\boxed{T_{x''}=T_{x},\qquad T_{y''}=T_{E_{1}''}^{-1}T_{y}^{-1}T_{x}^{4}}\label{eq:T-relation2}
\end{gather}
with $T_{E_{1}''}=\mathrm{p}_{23}\mathrm{p}_{45}T_{E_{1}}\mathrm{p}_{23}\mathrm{p}_{45}$ in $\operatorname{End}\big(H^{3}(X_{b_{o}}^{*},\mathbb{Z})\big)$, where $\mathrm{p}_{ij}$ are the permutation matrices. Since the generators of the cone $\Sigma_{o_{3}}$ are given by $N_{1}''=\log T_{y''}$ and $N_{2}''=\log T_{x''}$, we can evaluate these as
\begin{gather}
N_{1}''=N_{1},\qquad N_{2}''=4N_{1}-N_{2}+\Delta_{2,0}'',\label{eq:Npp-and-N}
\end{gather}
where $\Delta_{2,0}''$ is given by $\Delta_{2,0}''=\mathrm{p}_{23}\mathrm{p}_{45}\Delta_{1,0}'\mathrm{p}_{23}\mathrm{p}_{45}$ with the vanishing property $\Delta_{2,0}''\vert_{W_{2}}=0$. As before, $\Delta_{2,0}''$ is a non-vanishing element. Hence, the nilpotent cones $\Sigma_{o_{1}}$ and~$\Sigma_{o_{3}}$ glue together along the common half line $\mathbb{R}_{\geq0}N_{1}$ but do not lie on the same plane. Now we generalize these properties in the following proposition.

\begin{Proposition}\label{prop:Seq-N-Cones}\quad \begin{enumerate}\itemsep=0pt
\item[{\rm (1)}] The matrix $\crho$ preserves the monodromy weight filtration
\begin{gather*}
W_{0}\subset W_{2}\subset W_{4}\subset W_{6}=H^{3}(X_{b_{o}}^{*},\mathbb{Q}).
\end{gather*}
\item[{\rm (2)}] The $($closures of the$)$ monodromy nilpotent cones $\Sigma_{o_{1}}^{(n)}$, $\Sigma_{o_{2}}^{(n)}$, $\Sigma_{o_{3}}^{(n)}$ glue sequentially as in
\begin{gather*}
\dots,\Sigma_{o_{2}}^{(1)},\,\Sigma_{o_{1}}^{(1)},\,\Sigma_{o_{3}}^{(1)},\,\Sigma_{o_{2}},\,
\Sigma_{o_{1}},\,\Sigma_{o_{3}},\,\Sigma_{o_{2}}^{(-1)},\,\Sigma_{o_{1}}^{(-1)},\,\Sigma_{o_{3}}^{(-1)},\dots\,.
\end{gather*}

\item[{\rm (3)}] The generators of the cones satisfy
\begin{alignat*}{3}
& \mathrm{(i)} & & (N_{1}(n),N_{2}(n))=(N_{1},N_{2})\left(\begin{matrix}0 & -1\\
1 & 4 \end{matrix}\right)^{3n}+(\Delta_{1,n},\Delta_{2,n}),&\\
& \mathrm{(ii)} && (N_{1}'(n),N_{2}'(n))=(N_{1}',N_{2}')\left(\begin{matrix}0 & -1\\
1 & 4 \end{matrix}\right)^{-3n}+(\Delta_{1,n}',\Delta_{2,n}'),& \\
& \mathrm{(iii)} \quad & & (N_{1}''(n),N_{2}''(n))=(N_{1}'',N_{2}'')\left(\begin{matrix}0 & -1\\
1 & 4 \end{matrix}\right)^{-3n}+(\Delta_{1,n}'',\Delta_{2,n}''),&
\end{alignat*}
where $\Delta_{i,n},\Delta_{i,n}',\Delta_{i,n}''\in\operatorname{End}\big(H^{3}(X_{b_{o}}^{*},\mathbb{Q})\big)$ and satisfy $\Delta_{i,n}\vert_{W_{2}}=\Delta_{i,n}'\vert_{W_{2}}=\Delta_{i,n}''\vert_{W_{2}}=0$.

\item[{\rm (4)}] The following relations glue the nilpotent cones in $(2)$ $($see Fig.~{\rm \ref{Fig3})}:
\begin{gather*}
N_{1}(n)=N_{1}''(n),\qquad N_{2}(n)=N_{2}'(n),\qquad N_{2}''(n)=N_{1}'(n-1).
\end{gather*}
\end{enumerate}\end{Proposition}

\begin{proof} (1) Recall that $\crho$ is defined by $\crho=\cvarphi_{13}\cvarphi_{32}\cvarphi_{21}.$ Each isomorphism $\cvarphi_{ij}$ preserves the monodromy weight filtrations defined for each boundary point~$o_{k}$ (see Proposition~\ref{prop:connect-phi-21}). Hence, $\crho\colon H^{3}(X_{b_{o}},\mathbb{Q})\to H^{3}(X_{b_{o}},\mathbb{Q})$ preserves the monodromy weight filtration as claimed.

(2) We have introduced the generators of the nilpotent cones $\Sigma_{o_{k}}^{(n)}$ by $N_{i}(n),N_{i}'(n)$ and $N_{i}''(n)$ for $k=1,2,3$, respectively, in Section~\ref{para:Sigma(n)-and-Ni(n)}. Then the claim follows from the properties~(3) and~(4) (see also Fig.~\ref{Fig3}).

(3) By the definition of $N_{i}(n)$, it is straightforward to calculate $N_{i}(1)$ as
\begin{gather*}
N_{i}(1)=\crho^{-1}N_{i}\crho=\begin{cases}
-4N_{1}+15N_{2}+\Delta_{1,1}, & i=1,\\
-15N_{1}+56N_{2}+\Delta_{2,1}, & i=2,
\end{cases}
\end{gather*}
where $\Delta_{i,1}$ satisfy $\Delta_{1,1}\vert_{W_{2}}=\Delta_{2,1}\vert_{W_{2}}=0$ on the subspace $W_{2}\subset H^{3}(X_{b_{o}}^{*},\mathbb{Q})$. We note the relation $\left(\begin{smallmatrix}0 & -1\\
1 & 4 \end{smallmatrix}\right)^{3}=\left(\begin{smallmatrix}-4 & -15\\
15 & 56 \end{smallmatrix}\right)$ and arrange the above relation into the claimed matrix form for $n=1$. Then we can obtain the claimed formula (i) for general $n$ (in the first line) by evaluating $(\crho^{-n}N_{1}\crho^{n},\crho^{-n}N_{2}\crho^{n})$ inductively. In the evaluation, we should note that $\crho^{-1}\Delta_{i,n-1}\crho\vert_{W_{2}}=0$ if $\Delta_{i,n-1}\vert_{W_{2}}=0$ since $\crho$ preserves the monodromy weight filtration. For the second formula (ii), we note the relation
\begin{gather*}
(N_{1}',N_{2}')=(N_{1},N_{2})\left(\begin{matrix}-1 & 0\\
4 & 1
\end{matrix}\right)+(\Delta_{1,0}',0)
\end{gather*}
obtained in Proposition \ref{prop:Sigma0+Sigma2}(1). Taking the conjugations $\crho^{-n}(\text{-})\crho^{n}$ on the both sides of this relation, and using the first formula (i) for $\crho^{-n}(N_{1},N_{2})\crho^{n}$, we have the claimed formula. In the derivation, we use the relation
\begin{gather*}
\left(\begin{matrix}0 & -1\\
1 & 4
\end{matrix}\right)^{3n}\left(\begin{matrix}-1 & 0\\
4 & 1
\end{matrix}\right)=\left(\begin{matrix}-1 & 0\\
4 & 1
\end{matrix}\right)\left(\begin{matrix}0 & -1\\
1 & 4
\end{matrix}\right)^{-3n}
\end{gather*}
and also the property $\crho^{-1}\Delta'_{i,n-1}\crho\vert_{W_{2}}=0$ if $\Delta'_{i,n-1}\vert_{W_{2}}=0$. For the third relation (iii),
calculations are similar but we need to use the relation $\left(\begin{smallmatrix}0 & -1\\
1 & 4
\end{smallmatrix}\right)^{3n}\left(\begin{smallmatrix}1 & 4\\
0 & -1
\end{smallmatrix}\right)=\left(\begin{smallmatrix}1 & 4\\
0 & -1
\end{smallmatrix}\right)\left(\begin{smallmatrix}0 & -1\\
1 & 4
\end{smallmatrix}\right)^{-3n}$.

(4) Since $N_{i}(n)$, $N_{i}'(n)$, $N_{i}''(n)$ are defined by the conjugation of $N_{i}(n-1)$, $N_{i}'(n-1)$ and $N_{i}''(n-1)$ by $\crho$, it is sufficient to show the equalities
\begin{gather*}
N_{1}=N_{1}'',\qquad N_{2}=N_{2}',\qquad N_{2}''(1)=N_{1}'.
\end{gather*}
The first two relations are verified already in Proposition~\ref{prop:Sigma0+Sigma2} and~(\ref{eq:Npp-and-N}). For the last relation, we evaluate $N_{2}''(1)=\crho^{-1}N_{2}''\crho$ directly verifying its equality to~$N_{1}'$.
\end{proof}

\begin{figure}[t]\centering
\includegraphics[width=95mm]{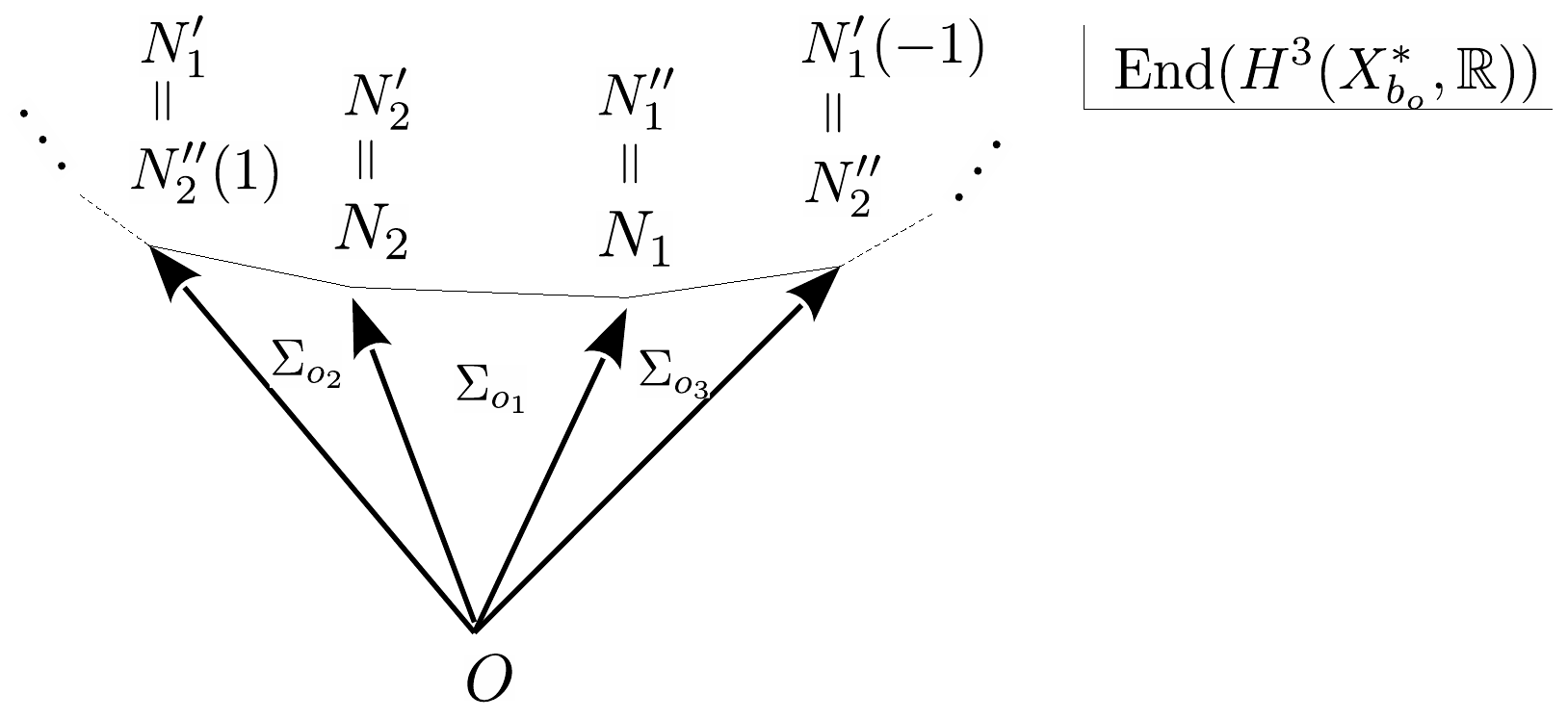}

\caption{Nilpotent cones glued in $\operatorname{End}\big(H^3(X_{b_o}^*,\mathbb{R})\big)$. Glueing continues to the both infinities, corresponding to the monodromy actions $\crho^n$, $n\to\pm\infty$. This should be compared with Fig.~\ref{Fig1}.}\label{Fig3}
\end{figure}

\begin{Corollary}\label{cor:XpartI} Consider the left ideal $\mathcal{I}_{2}:=\big\{ X\in\operatorname{End}\big(H^{3}(X_{b_{0}}^{*},\mathbb{R})\big)\,|\, X\vert_{W_{2}}=0\big\} $
of \linebreak $\operatorname{End}\big(H^{3}(X_{b_{0}}^{*},\mathbb{R})\big)$, and $\pi\colon\operatorname{End}\big(H^{3}(X_{b_{0}}^{*},\mathbb{R})\big)\to\operatorname{End}\big(H^{3}(X_{b_{0}}^{*},\mathbb{R})\big)/\mathcal{I}_{2}$
be the natural projection as a vector space. Then, taking the closure in $\operatorname{End}\big(H^{3}(X_{b_{0}}^{*},\mathbb{R})\big)/\mathcal{I}_{2}$, we have
\begin{gather*}
\bigcup_{n}\overline{\pi\big(\crho^{-n}\big(\Sigma_{o_{2}}\cup\Sigma_{o_{1}}\cup\Sigma_{o_{3}}\big)\crho^{n}\big)}=\mathbb{R}_{>0} \bar{c}_{1}+\mathbb{R}_{>0} \bar{c}_{2},
\end{gather*}
where $\bar{c}_{1}=-\bar{N}_{1}+(2+\sqrt{3})\bar{N}_{2}$ and $\bar{c}_{2}=\bar{N}_{1}-(2-\sqrt{3})\bar{N}_{2}$ with $\bar{N}_{i}=\pi(N_{i})$, $i=1,2$. \end{Corollary}
\begin{proof}
From Proposition \ref{prop:Seq-N-Cones}, we have
\begin{gather*}
\overline{\pi\left(\Sigma_{o_{1}}\cup\Sigma_{o_{2}}\cup\Sigma_{o_{3}}\right)} =\mathbb{R}_{\geq0}\pi(N_{1}')+\mathbb{R}_{\geq0}\pi(N_{2}'')
 =\mathbb{R}_{\geq0}\pi(4N_{2}-N_{1})+\mathbb{R}_{\geq0}\pi(4N_{1}-N_{2}).
\end{gather*}
Evaluating the matrix power $\left(\begin{smallmatrix}0 & -1\\
1 & 4\end{smallmatrix}\right)^{3n}$, it is easy to see that
\begin{gather*}
\lim_{n\to\infty}\mathbb{R}_{\geq0}\pi(N_{1}(n))=\lim_{n\to\infty}\mathbb{R}_{\geq0}\pi(N_{2}(n))=\mathbb{R}_{\geq0} \bar{c}_{1}
\end{gather*}
and
\begin{gather*}
\lim_{n\to-\infty}\mathbb{R}_{\geq0}\pi(N_{1}(n))=\lim_{n\to-\infty}\mathbb{R}_{\geq0}\pi(N_{2}(n))=\mathbb{R}_{\geq0} \bar{c}_{2}.
\end{gather*}
Then the claim follows from the gluing property (1) of Proposition~\ref{prop:Seq-N-Cones}.
\end{proof}

\subsection[Flopping curves and $T_{E_{1}}$]{Flopping curves and $\boldsymbol{T_{E_{1}}}$} \label{sub:TE1-Reye}

The matrix $T_{E_{1}}$ arises from the tangential intersection of the relevant components of the discriminant ${\rm Dis}$ in the moduli space $\overline{\mathcal{M}}_{X^{*}}^{\rm cpx}$. As noted in the remark above, $T_{E_{1}}$ may be identified with the first order correction from the quantum cohomology of $X_{1}$. To see this, let us introduce
\begin{gather}
N_{1}^{\mathtt{f}}:=\log\big(T_{x}^{-1}T_{y}^{4}\big)=4N_{2}-N_{1}\label{eq:N-Nf}
\end{gather}
and $N_{2}^{\mathtt{f}}=N_{2}'=N_{2}$. Here, we should note the difference in $N_{1}^{\mathtt{f}}$ from the definition $N_{1}'=\log\big(T_{E_{1}}^{-1}T_{x}^{-1}T_{y}^{4}\big)$.
\begin{Proposition}\label{prop: Cijk'-Cf-prop}Define $C_{ijk}'$ and $C_{ijk}^{\mathtt{f}}$ by $N_{i}'N_{j}'N_{k}'=C_{ijk}'\mathtt{N}_{0}$ and $N_{i}^{\mathtt{f}}N_{j}^{\mathtt{f}}N_{k}^{\mathtt{f}}=C_{ijk}^{\mathtt{f}}\mathtt{N}_{0}$ with $\mathtt{N}_{0}$ as given in Proposition~{\rm \ref{prop:Nijk-N0}}. Non-vanishing $($totally symmetric$)$ $C_{ijk}'$ and $C_{ijk}^{\mathtt{f}}$
are given by
\begin{gather}
(C_{111}',C_{112}',C_{122}',C_{222}')=(5,10,10,5),\nonumber\\
(C_{111}^{\mathtt{f}},C_{112}^{\mathtt{f}},C_{122}^{\mathtt{f}},C_{222}^{\mathtt{f}})=(-45,10,10,5).\label{eq:Cijk'-Cf}
\end{gather}
\end{Proposition}
\begin{proof}We derive these numbers by direct calculations of matrix products.
\end{proof}

The nilpotent matrices $N_{1}'$, $N_{2}'$ follow from the B-structure at $o_{2}$, which has been identified with the A-structure of~$X_{2}$. Hence the first equality in~(\ref{eq:Cijk'-Cf}) is a consequence from mirror symmetry. To see more details of the equality, let us recall the so-called mirror map which are defined by
\begin{gather}
t_{i}'=\frac{\int_{A_{i}'}\Omega_{\bm{x}'}}{\int_{A_{0}'}\Omega_{\bm{x}'}},\qquad t_{i}=\frac{\int_{A_{i}}\Omega_{\bm{x}}}{\int_{A_{0}}\Omega_{\bm{x}}}\label{eq:mirror-map}
\end{gather}
for each boundary point $o_{2}$ and $o_{1}$, respectively. If we relate these local definitions by the isomorphism $(\,^{t}\varphi_{b_{o}'b_{o}})^{-1}\colon H_{3}(X_{b_{o}}^{*},\mathbb{Z})\to H_{3}(X_{b_{o}'}^{*},\mathbb{Z})$ along the path $p_{b_{o}'\leftarrow b_{o}}$ (cf.\ Proposition~\ref{prop:connect-phi-21}), we have
\begin{gather*}
t_{1}'=-t_{1},\qquad t_{2}'=4t_{1}+t_{2}.
\end{gather*}

\begin{Proposition}\label{prop:Yukawa-Reye}Let $C_{ijk}$ be as defined in Proposition~{\rm \ref{prop:Nijk-N0}}. Also set $q_{1}':=e^{t_{1}'}$ and $q_{1}=e^{t_{1}}$. Then we have the following relations
\begin{gather*}
C_{ijk}^{\mathtt{f}}=\sum_{l,m,n}C_{lmn}\frac{dt_{l}}{dt_{i}'}\frac{dt_{m}}{dt_{j}'}\frac{dt_{n}}{dt_{k}'}
\end{gather*}
and
\begin{gather}
C_{111}'+50\frac{q_{1}'}{1-q_{1}'}=C_{111}^{\mathtt{f}}+50\frac{q_{1}}{1-q_{1}}\left(\frac{dt_{1}}{dt_{1}'}\right)^{3}.\label{eq:Cijk-Flop-inv}
\end{gather}
\end{Proposition}
\begin{proof}It is easy to verify these. For the second relation, we note that $50\frac{q_{1}}{1-q_{1}}\big(\frac{dt_{1}}{dt_{1}'}\big)^{3}$
$=50+50\frac{q_{1}'}{1-q_{1}'}$ for $q_{1}=1/q_{1}'$.
\end{proof}

The equality (\ref{eq:Cijk-Flop-inv}) is a consequence of the flop invariance of the quantum cohomology (see, e.g.,~\cite{Iri,IriXia, LLWang}). As mentioned in Remark~\ref{rem:Remark-TE1}, the number 50 represents the flopping curves. Comparing this with Proposition~\ref{prop: Cijk'-Cf-prop}, we see that the monodromy $T_{E_{1}}$ encodes the data of the flopping curves which is in the first order of the quantum cohomology of~$X_{1}$.

\subsection{Prepotentials}

The flop invariance expressed in (\ref{eq:Cijk-Flop-inv}) is known more precisely as the invariance of quantum cohomology under analytic continuations, where all higher order quantum corrections are taken into account. Here we rephrase this property as a property of the so-called prepotentials.

For the B-structure at each boundary point, we can define the prepotential. For example for the B-structure at $o_{1}$ and $o_{2}$, respectively,
they are given by
\begin{gather*}
\mathcal{F}=\frac{1}{2}\sum_{i=0}^{3}\int_{A_{i}}\Omega_{\bm{x}}\int_{B_{i}}\Omega_{\bm{x}},\qquad \mathcal{F}'=\frac{1}{2}\sum_{i=0}^{3}\int_{A_{i}'}\Omega_{\bm{x}'}\int_{B_{i}'}\Omega_{\bm{x}'}
\end{gather*}
with the symplectic integral bases for period integrals in $\Pi(x,y)$ and $\Pi'(x',y')$.
\begin{Proposition}\label{prop:prepot-F}By the isomorphism $(\,^{t}\varphi_{b_{o}'b_{o}})^{-1}\colon H_{3}(X_{b_{o}}^{*},\mathbb{Z})\to H_{3}(X_{b_{o}'}^{*},\mathbb{Z})$ along the path $p_{b_{o}'\leftarrow b_{o}}$ chosen as in Proposition~{\rm \ref{prop:connect-phi-21}}, $\mathcal{F}$ and $\mathcal{F}'$ are related by
\begin{gather*}
\mathcal{F}'=\mathcal{F}+\frac{1}{2}\sum_{i,j=1}^{2}Q_{ij}\int_{A_{i}}\Omega_{\bm{x}}\int_{A_{j}}\Omega_{\bm{x}},
\end{gather*}
where $(Q_{ij})=\left(\begin{smallmatrix}-25 & -2\\
2 & 0
\end{smallmatrix}\right)$.\end{Proposition}
\begin{proof}Using the basis $\{ A_{i},B_{j}\} $, $\{ A_{i}',B_{j}'\} $, the connection matrix has the form
\begin{gather*}
\big({}^{t}\varphi_{b_{o}'b_{0}}\big)^{-1}=\left(\begin{smallmatrix}-1 & 0 & 0 & 0 & 0 & 0\\
0 & 1 & 0 & 0 & 0 & 0\\
0 & -4 & -1 & 0 & 0 & 0\\
0 & 2 & 0 & -1 & 0 & 0\\
0 & -25 & -2 & -4 & 1 & 0\\
0 & 0 & 0 & 0 & 0 & -1
\end{smallmatrix}\right),
\end{gather*}
which gives the analytic continuation by $\Pi'(x',y')=(\,^{t}\varphi_{b_{o}'b_{0}})^{-1}\Pi(x,y)$.
From this, we read $A_{0}'=-A_{0}$, $B_{0}'=-B_{0}$ and
\begin{gather*}
\left(\begin{smallmatrix}A_{1}'\\
A_{2}'
\end{smallmatrix}\right)=R\left(\begin{smallmatrix}A_{1}\\
A_{2}
\end{smallmatrix}\right),\qquad \left(\begin{smallmatrix}B_{1}'\\
B_{2}'
\end{smallmatrix}\right)=\big(\,^{t}R\big)^{-1}\left(\begin{smallmatrix}B_{1}\\
B_{2}
\end{smallmatrix}\right)+\left(\begin{smallmatrix}-25 & -2\\
2 & 0
\end{smallmatrix}\right)\left(\begin{smallmatrix}A_{1}\\
A_{2}
\end{smallmatrix}\right),
\end{gather*}
where we set $R=\left(\begin{smallmatrix}1 & 0\\
-4 & -1
\end{smallmatrix}\right)$. The claimed formula is immediate from these.
\end{proof}

As we can deduce in the above proof, the prepotentials are invariant only up to quadratic terms of the $A_{i}$-periods under the analytic continuations even if they are symplectic and also preserve the monodromy weight filtrations. However the so-called Yukawa couplings are invariant since they are given by the third derivatives of the prepotentials with respect to the coordinates~$t_{i}$ (see~(\ref{eq:mirror-map})).

\section{Gluing monodromy nilpotent cones II}\label{section5}

We will study the following Calabi--Yau threefold of complete intersections:
\begin{gather*}
X=\left(\begin{smallmatrix}\mathbb{P}^{3}|\,2\,1\,1\\
\mathbb{P}^{3}|\,2\,1\,1
\end{smallmatrix}\right)^{2,66}.
\end{gather*}
We assume the defining equations of $X$ are chosen general unless otherwise mentioned. For such $X$, there is no other birational model than $X$. However $X$ has an interesting birational automorphism of infinite order \cite{Oguiso1}, and also has a non-trivial movable cone similar to the one in the preceding section.

\subsection{Birational automorphisms of infinite order}

Let $\pi_{i}\colon X\to\mathbb{P}^{3}$ be the projections to the first and second factor of $\mathbb{P}^{3}\times\mathbb{P}^{3}$ for $i=1$ and $2$, respectively. It is easy to see that the projection $\pi_{i}$ is surjective and generically $2:1$. We consider the Stein factorization $X\to W_{i}\to\mathbb{P}^{3}$ of the morphism $\pi_{i}\colon X\to\mathbb{P}^{3}$ and denote the morphism by $\phi_{i}\colon X\to W_{i}$ for $i=1,2$.
\begin{Proposition} For $i=1,2$, the morphism $W_{i}\to\mathbb{P}^{3}$ is a double cover of $\mathbb{P}^{3}$ branched along an octic, and $W_{i}$ is a $($smooth$)$ Calabi--Yau threefold. \end{Proposition}
\begin{proof} We omit proofs since they are standard (see, e.g., \cite{Oguiso1}).
\end{proof}

Let $\tilde{\tau}_{i}\colon W_{i}\simeq W_{i}^{+}$ be the deck transformation of the covering $W_{i}\to\mathbb{P}^{3}$. Then we have the map $\tau_{i}$ which covers $\tilde{\tau}_{i}$ as in the following diagram:
\begin{gather*}
\begin{xy}
(0,0)*++{X}="cX",
(-15,0)*++{X}="lX",
( 15,0)*++{X}="rX",
(-10,-20)*++{\mathbb{P}^3}="lP",
( 10,-20)*++{\mathbb{P}^3}="rP",
(-10,-10)*+{W_2}="lW",
(-26,-10)*+{W_2^+}="lWp",
( 10,-10)*+{W_1}="rW",
( 26,-10)*+{W_1^+.}="rWp",
( 17.5,-8)*++{\,^{\tilde\tau_1}},
(-17.5,-8)*++{\,^{\tilde\tau_2}},
( 17.5,-10)*++{\simeq},
(-17.5,-10)*++{\simeq},
( 7,1)*++{\sim},
( 7,3)*++{\,_{\tau_1}},
(-7,1)*++{\sim},
(-7,3)*++{\,_{\tau_2}},
( 7,-4)*++{\,_{\phi_1}},
(-7,-4)*++{\,_{\phi_2}},
\ar@{-->}_{} "cX";"lX"
\ar@{-->}^{} "cX";"rX"
\ar^{\pi_2} "cX";"lP"
\ar_{\pi_1} "cX";"rP"
\ar "cX";"lW"
\ar_{} "lW";"lP" 
\ar "lX";"lWp"
\ar_{2:1} "lWp";"lP"
\ar "cX";"rW"
\ar "rW";"rP"
\ar^{2:1} "rWp";"rP" 
\ar "rX";"rWp"
\end{xy}
\end{gather*}

\begin{Proposition}\label{prop:P3P3-curves}The following hold:
\begin{enumerate}\itemsep=0pt
\item[{\rm (i)}] The map $\tau_{i}\colon X\dashrightarrow X$ is birational but not bi-holomorphic.
\item[{\rm (ii)}]The morphism $\phi_{i}\colon X\to W_{i}$ contracts $80$ lines and $4$~conics to points, and the birational map $\tau_{i}$ is an Atyah's flop of these curves.
\end{enumerate}\end{Proposition}
\begin{proof}
(i), (ii) See the reference \cite[Proposition~6.1]{Oguiso1}. \end{proof}

\begin{Proposition} $(1)$ $\operatorname{Bir}(X)=\operatorname{Aut}(X)\cdot\langle\tau_{1},\tau_{2}\rangle$. $(2)$ $\tau_{i}^{2}={\rm id}$ for $i=1,2$. Also $\tau_{1}\tau_{2}$ has infinite order. \end{Proposition}
\begin{proof}
See \cite[Lemma~6.4]{Oguiso1}.
\end{proof}

\subsection[Mirror family of $X$]{Mirror family of $\boldsymbol{X}$}

We can describe the mirror family $\mathfrak{X}^{*}\to\mathcal{M}_{X^{*}}^{\rm cpx}$ of $X$ by writing $X$ in terms of a Gorenstein cone following Batyrev--Borisov. The parameter space of the defining equations up to isomorphisms naturally gives the moduli space $\mathcal{M}_{X^{*}}^{\rm cpx}$, which turns out to be compactified to $\mathbb{P}^{2}$ as before. Here we will not go into the details of the mirror family, but we only write the form of the Picard--Fuchs differential operator in the affine coordinate $[1,x,y]\in\overline{\mathcal{M}}_{X^{*}}^{\rm cpx}=\mathbb{P}^{2}$.
\begin{Proposition} Picard--Fuchs equations of the family on the affine coordinate $[1,x,y]$ are given by $\mathcal{D}_{1}w(x,y)=\mathcal{D}_{2}w(x,y)=0$ with
\begin{gather*}
\mathcal{D}_{1}= \big(3\theta_{x}^{2}-4\theta_{x}\theta_{y}+3\theta_{y}^{2}\big)-(\theta_{x}+\theta_{y})(2\theta_{x}+2\theta_{y}-1)(10x+6y)\\
\hphantom{\mathcal{D}_{1}=}{} +4\theta_{x}(2\theta_{x}+2\theta_{y}-1)(x-y),\\
\mathcal{D}_{2}= \big(\theta_{x}^{3}-\theta_{x}^{2}\theta_{y}+\theta_{x}\theta_{y}^{2}-\theta_{y}^{3}\big)-2(\theta_{x}+\theta_{y})^{2}(2\theta_{x}+2\theta_{y}-1)(x-y),
\end{gather*}
where $\theta_{x}=x\frac{\partial\;}{\partial x},\theta_{y}=y\frac{\partial\;}{\partial y}$. The discriminant locus of this system is given by ${\rm Dis}=D_{1}\cup D_{2}\cup D_{3}\cup {\rm Dis}_{0}$ with
\begin{gather*}
{\rm Dis}_{0}=\big\{ (1-4x-4y)^{4}-128xy\big(17+56(x+y)+16\big(x^{2}+y^{2}\big)\big)=0\big\} ,
\end{gather*}
and the coordinate lines $D_{i}$ of $\mathbb{P}^{2}$. \end{Proposition}
\begin{proof}The differential operators $\mathcal{D}_{1}$ and $\mathcal{D}_{2}$ arise from the Gel'fand--Kapranov--Zelevinski system after finding suitable factorizations of differential operators. See~\cite{HKTY} for more details. Once~$\mathcal{D}_{1}$ and~$\mathcal{D}_{2}$ are determined,
it is straightforward to determine the discriminant locus (singular locus) of the system from the equations of the characteristic variety.
\end{proof}

From the forms of $\mathcal{D}_{1}$and $\mathcal{D}_{2}$, the origin $x=y=0$ is expected to be a LCSL. In fact, we can verify all the properties for the LCSL in Definition~\ref{def:LCSL}. We also verify that there is no other LCSL point in $\overline{\mathcal{M}}_{X^{*}}^{\rm cpx}=\mathbb{P}^{2}$. In Fig.~\ref{Fig4}, we schematically describe the structure of the moduli space $\overline{\mathcal{M}}_{X^{*}}^{\rm cpx}.$ There, as in the preceding example, we see that the component ${\rm Dis}_{0}$ intersects tangentially with the divisors $D_{1}=\{ x=0\} $ and $D_{2}=\{ y=0\} $. This time, we blow-up at these two intersection points successively four times to make the intersections normal crossing (see Fig.~\ref{Fig4} right).
\begin{Remark} As in the previous example, we should be able to arrive at the mirror family $\mathfrak{X}^{*}\to\mathcal{M}_{X^{*}}^{\rm cpx}$ starting with a special family $\{ X_{\rm sp}\} _{a,b}$. But we leave this task for other occasions, since we have the mirror family in any case as above.
\end{Remark}

\begin{figure}[t]\centering

\includegraphics[width=95mm]{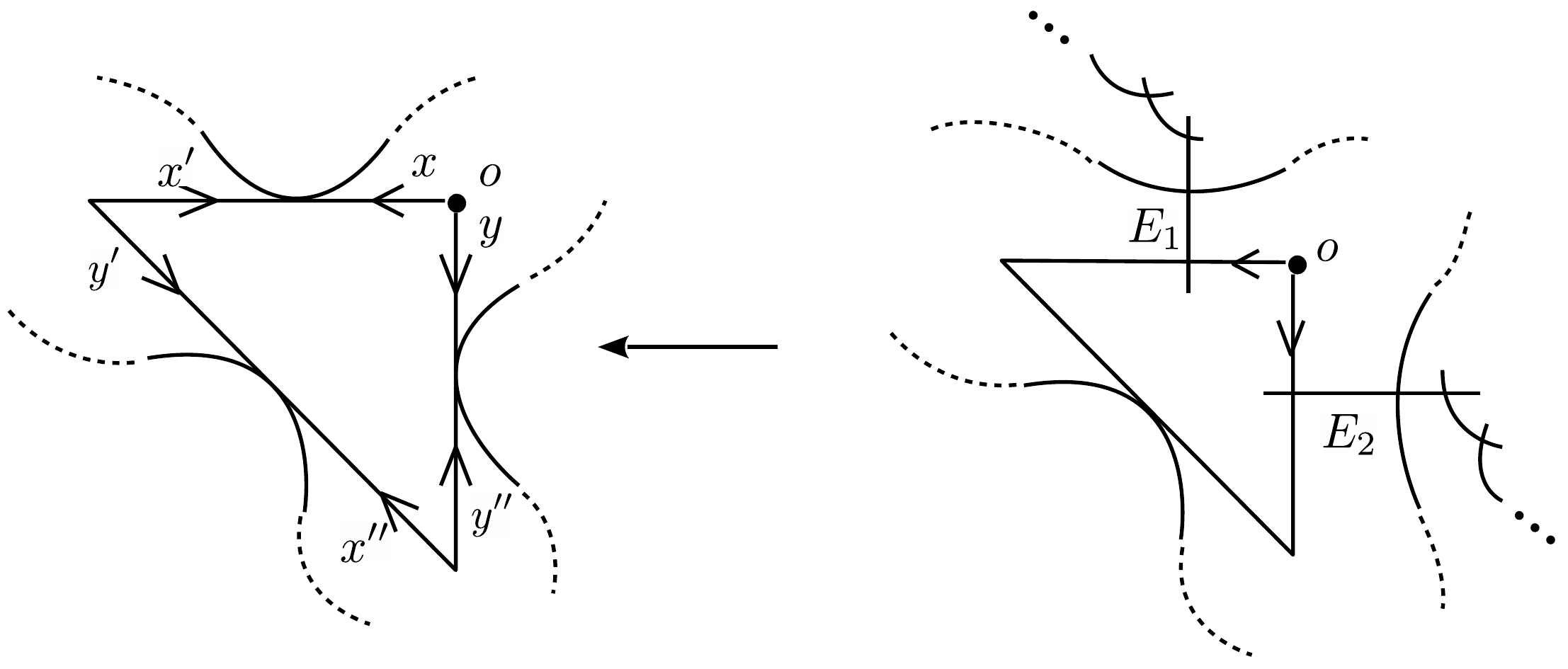}

\caption{Blowing-up $\overline{\mathcal M}_{X^*}^{\rm cpx}=\mathbb{P}^2$. There is only one LCSL at~$o_1$ in this case, and we blow-up four times at the two points introducing exceptional divisors $E_1$ and $E_2$ shown.}\label{Fig4}
\end{figure}

\subsection[B-structure of $X$ at the origin $o$]{B-structure of $\boldsymbol{X}$ at the origin $\boldsymbol{o}$} \label{sub:B-structure-P3P3}

As in Section~\ref{sub:B-str-Reye}, the canonical integral and symplectic structure can be introduced from the power series solution
\begin{gather*}
w_{0}(x,y)=\sum_{n,m}\frac{\Gamma(1+2n+2m)\Gamma(1+n+m)^{2}}{\Gamma(1+n)^{4}\Gamma(1+m)^{4}}x^{n}y^{m}
\end{gather*}
around the origin $o:=[1,0,0]$. Simply replacing the necessary parameters in the general formula \cite[Section~6.3.1]{HTcmp}, and fixing the so-called quadratic ambiguities there by $C_{kl}=0$, we obtain the canonical integral and symplectic structure in the form of period integrals $\Pi(x,y)$ with the corresponding symplectic basis
\begin{gather*}
\{ A_{0},A_{1},A_{2},B_{2},B_{1},B_{0}\} \subset H_{3}(X_{b_{o}}^{*},\mathbb{Z})\qquad \text{with}\qquad A_{i}\cap B_{j}=\delta_{ij},A_{i}\cap A_{j}=B_{i}\cap B_{j}=0,
\end{gather*}
where a base point~$b_{o}$ is taken near the origin. We denote by $\mathtt{T}_{x}$ the matrix of the monodromy transformation of~$\Pi(x,y)$ for a small loop around the divisor $D_{1}= \{ x=0 \} $, and similarly denote by $\mathtt{T}_{y}$ for a small loop around $D_{2}= \{ y=0 \} $. Writing the local solutions explicitly, it is straightforward to have
\begin{gather*}
\mathtt{T}_{x}=\left(\begin{smallmatrix}1 & 0 & 0 & 0 & 0 & 0\\
1 & 1 & 0 & 0 & 0 & 0\\
0 & 0 & 1 & 0 & 0 & 0\\
3 & 6 & 6 & 1 & 0 & 0\\
1 & 2 & 6 & 0 & 1 & 0\\
-4 & -1 & -3 & 0 & -1 & 1
\end{smallmatrix}\right),\qquad \mathtt{T}_{y}=\left(\begin{smallmatrix}1 & 0 & 0 & 0 & 0 & 0\\
0 & 1 & 0 & 0 & 0 & 0\\
1 & 0 & 1 & 0 & 0 & 0\\
1 & 6 & 2 & 1 & 0 & 0\\
3 & 6 & 6 & 0 & 1 & 0\\
-4 & -3 & -1 & -1 & 0 & 1
\end{smallmatrix}\right).
\end{gather*}
We introduce the dual basis $\mathcal{B}= \{ \alpha_{i},\beta_{i} \} $ of $H^{3}(X_{b_{o}}^{*},\mathbb{Z})$ and consider the dual actions $T_{x}:=(\,^{t}\mathtt{T}_{x})^{-1}$ and $T_{y}:=(\,^{t}\mathtt{T}_{y})^{-1}$ on $H^{3}(X_{b_{o}}^{*},\mathbb{Z})$ which are clearly unipotent.
\begin{Definition}
We define the monodromy nilpotent cone at $o$ by
\begin{gather*}
\Sigma_{o}=\left\{ \sum\lambda_{i}N_{i}\,|\,\lambda_{i}\geq0\right\} \subset\operatorname{End}\big(H^{3}(X_{b_{o}}^{*},\mathbb{Z})\big)
\end{gather*}
with $N_{1}=-\,^{t}(\log\mathtt{T}_{x})$ and $N_{2}=-\,^{t}(\log\mathtt{T}_{y})$.
\end{Definition}

Using the explicit forms of these matrices, we verify the following properties:
\begin{Proposition}\quad
\begin{enumerate}\itemsep=0pt \item[{\rm (1)}] The nilpotent element $N_{\lambda}:=\sum_{i}\lambda_{i}N_{i}$, $\lambda_{i}>0$, defines the weight monodromy filtration $W_{0}\subset W_{2}\subset W_{4}\subset W_{6}=H^{3}(X_{b_{0}}^{*},\mathbb{Z})$ with the same form $W_{2i}$ as given in~\eqref{eq:W2i-filt}.

\item[{\rm (2)}] We have
\begin{gather*}
N_{i}N_{j}N_{k}=C_{ijk}\mathtt{N_{0}}
\end{gather*}
with totally symmetric $C_{ijk}$ defined by $C_{111}=C_{222}=2$, $C_{122}=C_{112}=6$, $C_{ijk}=0$ otherwise, and $\mathtt{N}_{0}=\left(\begin{smallmatrix}0 & 1\\
O_{5} & 0
\end{smallmatrix}\right)$ where $O_{5}$ is the zero matrix of size $5\times5$.
\end{enumerate}
\end{Proposition}

From the above proposition, we see that the origin satisfies the conditions for LCSL. Also, looking other boundary points in the moduli space
$\overline{\mathcal{M}}_{X^{*}}^{\rm cpx}$, we see that no other LCSL exists in~$\overline{\mathcal{M}}_{X^{*}}^{\rm cpx}$.

\subsection[Gluing the monodromy nilpotent cone $\Sigma_{o}$]{Gluing the monodromy nilpotent cone $\boldsymbol{\Sigma_{o}}$}

Although there is only one LCSL in the mirror family $\mathfrak{X}^{*}\to\mathcal{M}_{X^{*}}^{\rm cpx}$, we can find the monodromy transformations which correspond to the birational automorphisms~$\tau_{1}$ and~$\tau_{2}$ of~$X$. We observe that the monodromy nilpotent cone $\Sigma_{o}$ extends to a larger cone (or cone structure) using these monodromy transformations, and we will identify the resulting cone structure with the movable cone~$\operatorname{Mov}(X)$ of~$X$.

\subsubsection[Path $p_{o\leftarrow E_i\leftarrow o}$, $i=1,2$]{Path $\boldsymbol{p_{o\leftarrow E_i\leftarrow o}}$, $\boldsymbol{i=1,2}$}

As shown in Fig.~\ref{Fig4}, the discriminant locus ${\rm Dis}$ has non-normal crossing intersection at three points. To make the intersections normal, we blow-up successively four times at the two points near the origin~$o$. We denote by~$E_{1}$ and~$E_{2}$, respectively, the exceptional divisors introduced by the blow-ups (see Fig.~\ref{Fig4} right). As we see in the form of the discriminant~${\rm Dis}$, the family over $\mathcal{M}_{X^{*}}^{\rm cpx}$ is symmetric under $x\leftrightarrow y$. Because of this symmetry reason, it suffices to describe the divisor $D_{2}= \{ y=0\} $ which intersects with~$E_{1}$ at $[1,x,y]=\big[1,\frac{1}{4},0\big]\in\mathbb{P}^{2}$. Explicitly, we introduce the blow-up coordinate at the origin $q_{12}:=E_{1}\cap D_{y}$ by
\begin{gather*}
s_{1}=4x-1,\qquad s_{2}=\frac{1}{2^{6}}\frac{y}{(1-4x)^{4}}.
\end{gather*}

\begin{Definition} Let $R_{12}=\big\{ \frac{1}{4}+\frac{s_{1}}{4}e^{i\theta}\,|\,0\leq\theta\leq2\pi\big\} $ be a small loop around $E_{1}$ on $D_{2}$. We denote by $p_{q_{12}\leftarrow b_{o}}= \{ (1-t)b_{o}+tq_{12}\,|\,0\leq t\leq1-\varepsilon \} $ the straight line connecting the base point $b_{o}$ near $o$ and a point $q_{12}$ on the small loop $R_{12}$. Then we define
\begin{gather*}
p_{b_{o}\leftarrow E_{1}\leftarrow b_{0}}:=(p_{q_{12}\leftarrow b_{o}})^{-1}\circ R_{12}\circ p_{q_{12}\leftarrow b_{o}}
\end{gather*}
to be the composite path which encircles the divisor $E_{1}$ from the base point $b_{o}$. In a similar way, we define a closed path $p_{b_{0}\leftarrow E_{2}\leftarrow b_{o}}$ which encircle the divisor~$E_{2}$ from $b_{0}$ (see Fig.~\ref{Fig4}).

\subsubsection[Monodromy around $E_i$]{Monodromy around $\boldsymbol{E_i}$} \label{para:monod-Ei-p3p3}

Let $(x',y')=\big(\frac{1}{x},\frac{y}{x}\big)$ be the affine coordinate with the origin $[0,1,0]\in\mathbb{P}^{2}$ and $b_{o}'$ be a base point near the origin. We denote by $\mathtt{T}_{x'}'$ and $\mathtt{T}_{y'}'$ the local monodromy around $x'=0$ and $y'=0$, respectively. Conjugating $\mathtt{T}_{x'}'$, $\mathtt{T}_{y'}'$ by the connection matrix for the path $p_{b_{o}'\leftarrow b_{0}}=p_{b_{0}'\leftarrow q_{12}}\circ p_{q_{12}\leftarrow b_{o}}$, we define the corresponding monodromy matrices $\mathtt{T}_{x'}$ and $\mathtt{T}_{y'}$ for loops with the base point $b_{o}$. We define $T_{x'}:=(^{t}\mathtt{T}_{x'})^{-1}$ and $T_{y'}:=(\,^{t}\mathtt{T}_{y'})^{-1}$ to be the linear actions on the dual space $H^{3}(X_{b_{o}}^{*},\mathbb{Z})$. \end{Definition}
\begin{Proposition}\label{prop:Tx'-Ty'}We have
\begin{gather*}
T_{x'}=\left(\begin{smallmatrix}-1 & -1 & \,\,3 & \,\,6 & -10 & \,\,2\\
\,0 & \,1 & -6 & -12 & \,\,8 & \,\,2\\
\,0 & \,0 & -1 & \,\,0 & \,12 & -6\\
\,0 & \,0 & \,\,0 & -1 & -6 & \,\,3\\
\,0 & \,0 & \,\,0 & \,\,0 & \,\,1 & -1\\
\,0 & \,0 & \,\,0 & \,\,0 & \,\,0 & -1
\end{smallmatrix}\right),\qquad T_{y'}=\left(\begin{smallmatrix}1 & \,\,0 & -1 & \,\,1 & \,\,3 & \,\,4\\
0 & \,\,1 & \,0 & -6 & -6 & -3\\
0 & \,\,0 & \,1 & -2 & -6 & -1\\
0 & \,\,0 & \,0 & \,\,1 & \,\,0 & \,\,1\\
0 & \,\,0 & \,0 & \,\,0 & \,\,1 & \,\,0\\
0 & \,\,0 & \,0 & \,\,0 & \,\,0 & \,\,1
\end{smallmatrix}\right).
\end{gather*}
In particular we have $T_{y'}=T_{y}$. \end{Proposition}

\begin{proof} These are based on explicit calculations. Here we only sketch the calculations. We first make local solutions using the coordinate $(s_{1},s_{2})$ centered at~$q_{12}$. Then their domain of convergence have overlap both with the local solutions around $(x,y)=(0,0)$ and $(x',y')=(0,0)$. Then it is straightforward to obtain the connection matrices. The local monodromy matrices $\mathtt{T}_{x'}'$ and~$\mathtt{T}_{y'}'$ are easily read off from the local solutions. Then by conjugating these local monodromy matrices by the connection matrix, we have the expressions for $\mathtt{T}_{x'}$, $\mathtt{T}_{y'}$ as the linear actions on $H_{3}(X_{b_{0}}^{*},\mathbb{Z})$. Translating these to $H^{3}(X_{b_{o}}^{*},\mathbb{Z})$, we obtain~$T_{x'}$ and~$T_{y'}$.
\end{proof}

Similarly we define the monodromy matrix $\mathtt{T}_{E_{1}}$ along the loop $p_{b_{o}\leftarrow E_{1}\leftarrow b_{o}}$ and set $T_{E_{1}}:=(\,^{t}\mathtt{T}_{E_{1}})^{-1}$. Corresponding to Proposition~\ref{prop:Picard-Lefschetz-Reye} we have
\begin{Proposition}[`Picard--Lefschetz formula' for the flopping curves] \label{prop:Picard--Lefschetz-P3P3}
\quad \begin{enumerate}\itemsep=0pt
\item[{\rm (1)}] The monodromy matrix is given by
\begin{gather*}
T_{E_{1}}=\left(\begin{smallmatrix}-1 & 0 & 0 & 0 & 0 & 0\\
0 & 1 & -6 & 0 & 48 & 0\\
0 & 0 & -1 & 0 & 0 & 0\\
0 & 0 & 0 & -1 & -6 & 0\\
0 & 0 & 0 & 0 & 1 & 0\\
0 & 0 & 0 & 0 & 0 & -1
\end{smallmatrix}\right).
\end{gather*}
In particular, this is quasi-unipotent.

\item[{\rm (2)}] For $T_{E_{1}}^{2}$ we have
\begin{gather*}
T_{E_{1}}^{2}=\left(\begin{smallmatrix}1\\
 & 1 & & & 96\\
 & & 1\\
 & & & 1\\
 & & & & 1\\
 & & & & & 1
\end{smallmatrix}\right),\qquad \text{i.e.},\qquad \begin{cases}
\alpha_{1}\to\alpha_{1}+96\beta_{1},\\
\beta_{1}\to\beta_{1},\\
\alpha_{i}=\alpha_{i}, \quad \beta_{i}=\beta_{i}, \quad i\not=1.
\end{cases}
\end{gather*}

\item[{\rm (3)}] By symmetry, we have similar formula for $T_{E_{2}}$ and $T_{E_{2}}^{2}$. In particular, $T_{E_{2}}^{2}$ is given by
$\alpha_{2}\to\alpha_{2}+96\beta_{2}$, $\beta_{2}\to\beta_{2}$, with $\alpha_{i}=\alpha_{i}$, $\beta_{i}=\beta_{i}$ for $i\not=2$.

\end{enumerate}\end{Proposition}
\begin{proof}These results follow from making local solutions and the analytic continuations of them. Again, calculations are straightforward since local solutions around $(x,y)=(0,0)$ and $(s_{1},s_{2})=(0,0)$ have overlap in their domains of convergence. \end{proof}

\begin{Remark} \label{rem:Remark-TE1-II}(1) As before, the monodromy action~(2) in the above proposition is expressed in terms of the symplectic basis $\{ A_{i},B_{j}\} $ of $H_{3}(X_{b_{o}}^{*},\mathbb{Z})$ as
\begin{gather*}
A_{1}\to A_{1},\qquad B_{1}\to B_{1}-96 A_{1}.
\end{gather*}

(2) We have seen in Proposition~\ref{prop:P3P3-curves} that each $\tau_{i}\colon X\dashrightarrow X$ is an Atiyah's flop with respect to 80~lines and also 4~conics. We observe that $96=80+4\times2^{2}$ holds for the number in $T_{E_{i}}^{2}$. We can verify the corresponding relations also for other examples. Based on these, we conjecture the following general form:
\begin{gather*}
A_{1}\to A_{1},\qquad B_{1}\to B_{1}-\big(n_{0}(1)+n_{0}(2)\times2^{2}\big) A_{1}
\end{gather*}
for the Atiyah's flops of $n_{0}(1)$ lines and $n_{0}(2)$ conics associated to the contractions to the double cover of $\mathbb{P}^{3}$.
\end{Remark}

\subsubsection{Monodromy relations}
Take affine coordinates $(x,y)$, $(x',y')$ and $(x'',y'')$ of $\mathbb{P}^{2}$ as shown in Fig.~\ref{Fig4}. Let $T_{x'}$, $T_{y'}$ be as defined in Proposition~\ref{prop:Tx'-Ty'}.
\begin{Proposition} \label{prop:Monod-rel-P3P3}The following monodromy relations holds
\begin{gather}
\boxed{T_{x'}=T_{E_{1}}^{-1}T_{x}^{-1}T_{y}^{3},\qquad T_{y'}=T_{y},\qquad T_{E_{1}}T_{y}=T_{y}T_{E_{1}}.}\label{eq:XP3P3-monod-rel-1}
\end{gather}
\end{Proposition}
\begin{proof} We have the second and the third relations since all the divisors are normal crossing after the blow-ups. We can verify the first relation directly by using $T_{x}=(\,^{t}\mathtt{T}_{x})^{-1}$, $T_{y}=(\,^{t}\mathtt{T}_{y})^{-1}$ given in Section~\ref{sub:B-structure-P3P3} and $T_{x'}$, $T_{E_{1}}$ in Section~\ref{para:monod-Ei-p3p3}.
\end{proof}

\begin{Definition}
Define the following conjugations of $T_{x},T_{y}$ by $T_{E_{1}}$:
\begin{gather*}
\tilde{T}_{x}:=T_{E_{1}}^{-1}T_{x}T_{E_{1}},\qquad \tilde{T}_{y}:=T_{E_{1}}^{-1}T_{y}T_{E_{1}}.
\end{gather*}
Using these, we define the monodromy nilpotent cone by
\begin{gather*}
\tilde{\Sigma}_{o}:=\left\{ \sum\lambda_{i}\tilde{N}_{i}\,|\,\lambda_{i}>0\right\} \subset\operatorname{End}\big(H^{3}(X_{b_{o}}^{*},\mathbb{R})\big),
\end{gather*}
where $\tilde{N}_{1}:=\log\tilde{T}_{x}$ and $\tilde{N}_{2}:=\log\tilde{T}_{y}$.
\end{Definition}

\begin{Proposition}\label{prop:Glue-P3P3-Tx}The $($closures of the$)$ monodromy nilpotent cones $\Sigma_{o}$ and $\tilde{\Sigma}_{o}$ glue along the ray $\mathbb{R}_{\geq0}N_{2}$, but they are not on the same two dimensional plane. \end{Proposition}
\begin{proof}Using the monodromy relations in Proposition \ref{prop:Monod-rel-P3P3}, we have $\tilde{T}_{y'}=T_{y}$. Hence the claim is immediate since we have $\tilde{N}_{2}=N_{2}$ by definition. To see the second claim, we use again the monodromy relations to have
\begin{gather*}
\tilde{T}_{x}=T_{E_{1}}^{-1}T_{x}T_{E_{1}}=T_{E_{1}}^{-1}T_{y}^{3}T_{x'}^{-1}=T_{E_{1}}^{-1}T_{x'}^{-1}T_{y}^{3},
\end{gather*}
which is reminiscent of the relation (\ref{eq:T-relation}). In fact, after some matrix calculations, we obtain
\begin{gather}
\tilde{N}_{1}=6N_{2}-N_{1}+\Delta_{1},\qquad \Delta_{1}=\left(\begin{smallmatrix}0 & 0 & 0 & 0 & 48 & -\frac{44}{3}\\
0 & 0 & 0 & 0 & -112 & 48\\
0 & 0 & 0 & 0 & 0 & 0\\
0 & 0 & 0 & 0 & 0 & 0\\
0 & 0 & 0 & 0 & 0 & 0\\
0 & 0 & 0 & 0 & 0 & 0
\end{smallmatrix}\right),\label{eq:tildeN1}
\end{gather}
where $\Delta_{1}$ satisfies $\Delta_{1}\vert_{W_{2}}=0$. Since the nilpotent cone $\Sigma_{o}$ lies on the plane spanned by~$N_{1}$ and~$N_{2}$, and $aN_{1}+bN_{2}\vert_{W_{2}}\not=0$ holds for any~$a$,~$b$, the basis element $\tilde{N}_{1}$ does not lie on the same plane as $\Sigma_{o}$.
\end{proof}

\subsubsection{Gluing nilpotent cones} As the example in the previous section, the structure of the moduli space $\mathcal{M}_{X^{*}}^{\rm cpx}$ is
symmetric under the exchange of $x$ and $y$. Hence, corresponding to~(\ref{eq:XP3P3-monod-rel-1}), we have
\begin{gather}
\boxed{T_{x''}=T_{x},\qquad T_{y''}=T_{E_{2}}^{-1}T_{y}^{-1}T_{x}^{3},\qquad T_{E_{2}}T_{x}=T_{x}T_{E_{2}}.}\label{eq:XP3P3-monod-rel-2}
\end{gather}
When we define $\tilde{T}_{x}':=T_{E_{2}}^{-1}T_{x}T_{E_{2}}$, $\tilde{T}_{y}':=T_{E_{2}}^{-1}T_{y}T_{E_{2}}$,
we have the following relations
\begin{gather*}
\tilde{N}_{1}'=N_{1},\qquad \tilde{N}_{2}'=6N_{1}-N_{2}+\Delta_{1}'
\end{gather*}
for $\tilde{N}_{1}':=\log\tilde{T}_{x}$, $\tilde{N}_{2}':=\log\tilde{T}_{y}$ with $\Delta'_{1}\vert_{W_{2}}=0$. This entails the corresponding
gluing property described in Proposition~\ref{prop:Sigma0+Sigma2}. We summarize these two actions into the following general form.
\begin{Definition}
We denote by $\tau_{E_{i}}$ the conjugations by $T_{E_{i}}$ on $\operatorname{End}\big(H^{3}(X_{b_{o}}^{*},\mathbb{Q})\big)$,
which act on the nilpotent matrices $N$ in general as
\begin{gather*}
\tau_{E_{1}}(N)=T_{E_{1}}^{-1}NT_{E_{1}},\qquad \tau_{E_{2}}(N)=T_{E_{2}}^{-1}NT_{E_{2}}.
\end{gather*}
We set $G:=\langle\tau_{E_{1}},\tau_{E_{2}}\rangle$, i.e., the group generated by $\tau_{E_{1}}$ and $\tau_{E_{2}}$.
\end{Definition}

\begin{Proposition}\label{prop:tau1-tau2-actions}\quad
\begin{enumerate}\itemsep=0pt
\item[{\rm (1)}] The actions of $\tau_{E_{i}}^{n}\in G$ on $N_{1}=\log T_{x}$, $N_{2}=\log T_{y}$ are summarized as
\begin{gather*}
\big(\tau_{E_{1}}^{n}(N_{1}),\tau_{E_{1}}^{n}(N_{2})\big)=(N_{1},N_{2})\left(\begin{matrix}-1 & 0\\
6 & 1
\end{matrix}\right)^{n}+(\Delta_{n},0),\\
\big(\tau_{E_{2}}^{n}(N_{1}),\tau_{E_{2}}^{n}(N_{2})\big)=(N_{1},N_{2})\left(\begin{matrix}1 & 6\\
0 & -1 \end{matrix}\right)^{n}+(0,\Delta_{n}'),
\end{gather*}
where $\Delta_{n}$, $\Delta_{n}'$ are elements in $\operatorname{End}\big(H^{3}(X_{b_{o}}^{*},\mathbb{Q})\big)$ satisfying $\Delta_{n}\vert_{W_{2}}=\Delta_{n}'\vert_{W_{2}}=0$. In particular, we have
\begin{gather*}
\tau_{E_{1}}^{n}(N_{2})=N_{2},\qquad \tau_{E_{2}}^{n}(N_{1})=N_{1}.
\end{gather*}

\item[{\rm (2)}] The action of $\sigma\in G$ on $\Delta_{n},\Delta_{n}'$ preserves the vanishing properties of $\Delta_{n}$, $\Delta_{n}'$ on
$W_{2}$, i.e., $\sigma(\Delta_{n})\vert_{W_{2}}=\sigma(\Delta_{n}')\vert_{W_{2}}=0$.

\item[{\rm (3)}] $\Delta_{n}$, $\Delta_{n}'$ have the following forms:
\begin{gather*}
\Delta_{2m}=\left(\begin{matrix}
O_{24} & \begin{smallmatrix}-96m & 0\\
0 & -96m
\end{smallmatrix}\\
O_{44} & O_{42}
\end{matrix}\right),\qquad \Delta_{2m-1}=\left(\begin{matrix}
O_{24} & \begin{smallmatrix}96\big(m-\frac{1}{2}\big) & -\frac{44}{3}\\
-122 & 96\big(m-\frac{1}{2}\big)
\end{smallmatrix}\\
O_{44} & O_{42}
\end{matrix}\right)
\end{gather*}
and $\Delta_{n}'=\mathrm{p}_{23}\mathrm{p}_{45}\Delta_{n}\mathrm{p}_{23}\mathrm{p}_{45}$, where $O_{ab}$ is the $a\times b$ zero matrix and $\mathrm{p}_{ij}$ represents the permutation matrix for the transposition $(i,j)$.
\end{enumerate}\end{Proposition}

\begin{proof} These properties are verified by explicit calculations using the matrix representa\-tions $T_{x}$,~$T_{y}$ and~$T_{E_{i}}$ given previous sections. The vanishing properties follow inductively from $\Delta_{1}\vert_{W_{2}}=\Delta_{1}'\vert_{W_{2}}=0$ and the fact that both $T_{E_{1}}$ and $T_{E_{2}}$ preserve the monodromy weight filtration $W_{0}\subset W_{2}\subset W_{4}\subset W_{6}=H^{3}(X_{b_{o}}^{*},\mathbb{Q})$.
\end{proof}

As before, let $\mathcal{I}_{2}:=\big\{ X\in\operatorname{End}\big(H^{3}(X_{b_{0}}^{*},\mathbb{R})\big)\,|\, X\vert_{W_{2}}=0\big\} $ be an left ideal of $\operatorname{End}\big(H^{3}(X_{b_{0}}^{*},\mathbb{R})\big)$, and $\pi\colon \operatorname{End}\big(H^{3}(X_{b_{0}}^{*},\mathbb{R})\big)\to\operatorname{End}\big(H^{3}(X_{b_{0}}^{*},\mathbb{R})\big)/\mathcal{I}_{2}$ be the natural projection. Since~$T_{E_{i}}$ preserve the monodromy weight filtration, and by the definition of $\tau_{E_{i}}$, it is easy to see that $\sigma(\mathcal{I}_{2})\subset\mathcal{I}_{2}$ for all $\sigma\in G$. Hence we have the naturally induced $G$ action on the quotient $\operatorname{End}\big(H^{3}(X_{b_{0}}^{*},\mathbb{R})\big)/\mathcal{I}_{2}$ by $\sigma(X+\mathcal{I}_{2}):=\sigma(X)+\mathcal{I}_{2}$. Note that, if we denote by $\bar{\sigma}$ the action of $\sigma\in G$ on the quotient space, we have $\overline{\sigma\tau}=\bar{\tau}\bar{\sigma}$ (i.e., anti-homomorphism by our convention for the adjoint action) for all $\sigma,\tau\in G$.

\begin{Corollary}\label{cor:Union-cones-P3P3}Denote by $\bar{N}_{i}:=\pi(N_{i})$ the basis of the cone $\pi(\Sigma_{o})$ in the quotient space. Then the following hold:
\begin{enumerate}\itemsep=0pt
\item[{\rm (1)}] Define $\tau_{12}:=\tau_{1}\tau_{2}$ with $\tau_{i}=\tau_{E_{i}}$. We have
\begin{gather*}
\big(\bar{\tau}_{12}^{n}(\bar{N}_{1}),\bar{\tau}_{12}^{n}(\bar{N}_{2})\big)=\big(\bar{N}_{1},\bar{N}_{2}\big)\left(\begin{matrix}35 & 6\\
-6 & -1
\end{matrix}\right)^{n}.
\end{gather*}

\item[{\rm (2)}] $\{ \sigma\in G\,|\,\bar{\sigma}(\bar{N}_{i})=\bar{N}_{i}, \, i=1,2 \} =\langle\tau_{1}^{2},\tau_{2}^{2}\rangle$.

\item[{\rm (3)}] Taking the closure in $\operatorname{End}\big(H^{3}(X_{b_{0}}^{*},\mathbb{R})\big)/\mathcal{I}_{2}$, we have
\begin{gather*}
\bigcup_{\sigma\in G}\overline{\pi(\sigma(\Sigma_{o}))}=\mathbb{R}_{>0}\big({-}\bar{N}_{1}+(3+2\sqrt{2})\bar{N}_{2}\big)+\mathbb{R}_{>0}\big(\bar{N}_{1}+(3-2\sqrt{2})\bar{N_{2}}\big).
\end{gather*}
\end{enumerate}\end{Corollary}
\begin{proof} The equality (1) follow from Proposition~\ref{prop:tau1-tau2-actions}(1) and $\bar{\tau}_{12}=\bar{\tau}_{2}\bar{\tau}_{1}$. By definition, $G$ is generated by $\tau_{1}$, $\tau_{2}$. Then the claim (2) follows from Proposition~\ref{prop:tau1-tau2-actions}(1) and the above equality~(1). To show the claim (3), we write by $(N_{1},N_{2})_{>0}$ the cone generated by $N_{1}$ and $N_{2}$. Then we first show that
the following cones successively glue together to a large cone:
\begin{gather*}
\big(\tau_{12}^{n}(N_{1}),\tau_{12}^{n}(N_{2})\big)_{>0},\qquad \big(\tau_{12}^{n}(\tau_{1}N_{1}),\tau_{12}^{n}(\tau_{1}N_{2})\big)_{>0}, \qquad n\in\mathbb{Z}.
\end{gather*}
Using the property $\tau_{1}(N_{2})=N_{2}$, $\tau_{2}(N_{1})=N_{1}$, we have
\begin{gather*}
\tau_{12}^{n}(\tau_{1}N_{1})=\tau_{12}^{n+1}(N_{1}),\qquad \tau_{12}^{n}(\tau_{1}N_{2})=\tau_{12}^{n}(N_{2}),
\end{gather*}
by which we can arrange a sequence of cones schematically as follows:
\begin{gather*}
\begin{array}{@{}ccccccc}
 & \big(\tau_{12}^{n+1}(\tau_{1}N_{1}), & \tau_{12}^{n+1}(\tau_{1}N_{2})\big)_{>0} & \big(\tau_{12}^{n}(\tau_{1}N_{1}), & \tau_{12}^{n}(\tau_{1}N_{2})\big)_{>0} & \cdots\\
 & \shortparallel & \shortparallel & \shortparallel & \shortparallel & \shortparallel\\
 & \cdots & \big(\tau_{12}^{n+1}(N_{2}), & \tau_{12}^{n+1}(N_{1})\big)_{>0} & \big(\tau_{12}^{n}(N_{2}), & \tau_{12}^{n}(N_{1})\big)_{>0}
\end{array}
\end{gather*}
Let us note that $\tau_{2}\tau_{1}=\tau_{2}^{2}\tau_{12}^{-1}\tau_{1}^{2}$ and $\tau_{2}=\tau_{2}^{2}\tau_{12}^{-1}\tau_{1}$ hold. Then, using these relations, we can deduce the decomposition
\begin{gather*}
G=\big\langle\tau_{12},\tau_{1}^{2},\tau_{2}^{2}\big\rangle\cup\big\langle\tau_{12},\tau_{1}^{2},\tau_{2}^{2}\big\rangle\tau_{1}.
\end{gather*}
Since $\tau_{1}^{2}$, $\tau_{2}^{2}$ have trivial actions on $\bar{N}_{i}$, $i=1,2$, the above sequence of the cones explain the union $\bigcup_{\sigma\in G}\overline{\pi(\sigma(\Sigma_{o}))}$. After some linear algebra of the matrix power $\left(\begin{smallmatrix}35 & 6\\
-6 & -1
\end{smallmatrix}\right)^{n}$, we can determine the infinite union in the claimed form.
\end{proof}

\begin{figure}[t]\centering
\includegraphics[width=102mm]{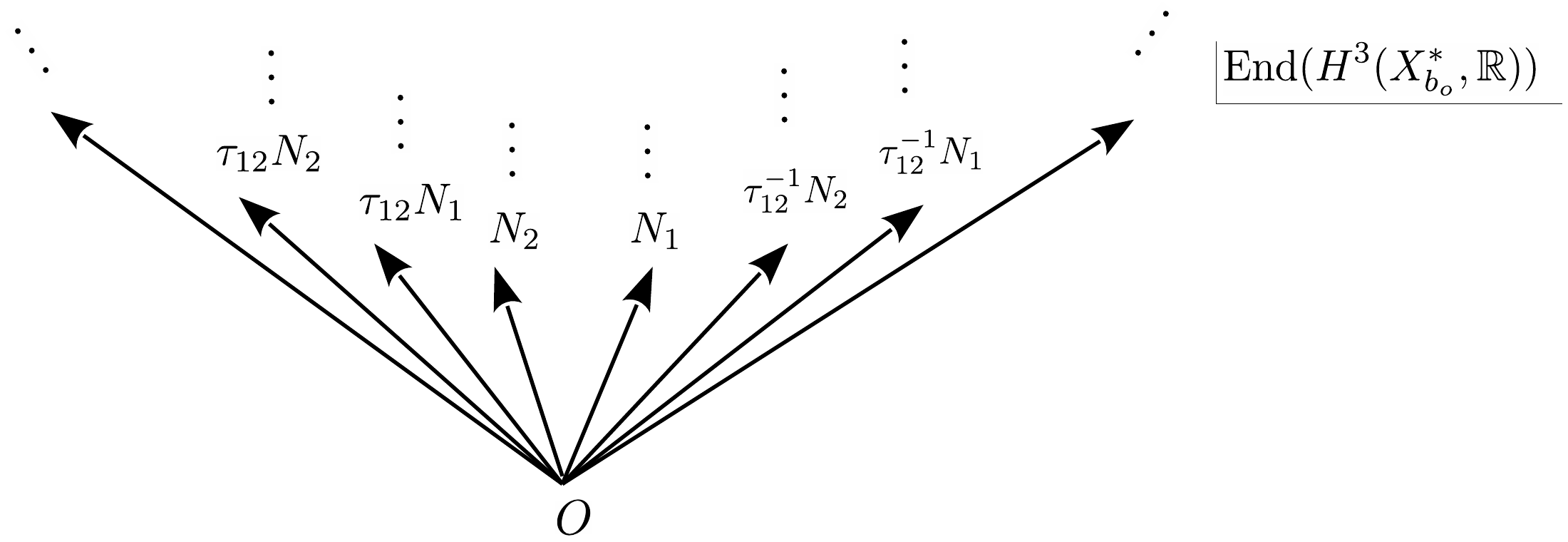}

\caption{Gluing nilpotent cones. The nilpotent cones $\sigma(\Sigma_o)$ glue in $\operatorname{End}\big(H^3(X_{b_o}^*,\mathbb{R})\big)$. The composite actions of $\tau_1^2$ and $\tau_2^2$ on each ray are not trivial, although they are trivial on the images in $\operatorname{End}\big(H^3(X_{b_o}^*,\mathbb{R})\big)/\mathcal{I}_2$.}\label{Fig5}
\end{figure}

\subsection[Flopping curves and $T_{E_{1}}$]{Flopping curves and $\boldsymbol{T_{E_{1}}}$}

The monodromy $T_{E_{1}}$ has appeared in the moduli space from the tangential intersection of the discriminants. This is quite parallel to Section~\ref{sub:TE1-Reye}. However, $T_{E_{1}}$ is not unipotent but only quasi-unipotent in the present case. This prevents a parallel definition to the second equation in~(\ref{eq:N-Nf}), but this time we set
\begin{gather*}
N_{1}^{\mathtt{f}}:=6N_{2}-N_{1}
\end{gather*}
with $\tilde{N}_{1}=N_{1}^{\mathtt{f}}+\Delta_{1}$ (see~(\ref{eq:tildeN1})) and also $\tilde{N}_{2}=N_{2}^{\mathtt{f}}=N_{2}$. Then we have
\begin{Proposition}
Let $\tilde{N}_{i}\tilde{N}_{j}\tilde{N}_{k}=\tilde{C}_{ijk}\mathtt{N}_{0}$ and $N_{i}^{\mathtt{f}}N_{j}^{\mathtt{f}}N_{k}^{\mathtt{f}}=C_{ijk}^{\mathtt{f}}\mathtt{N}_{0}$ with $\mathtt{N}_{0}$ as given in Proposition~{\rm \ref{prop:Nijk-N0}}. Non-vanishing $($totally symmetric$)$ $\tilde{C}_{ijk}$ and $C_{ijk}^{\mathtt{f}}$ are given by
\begin{gather}
\big(\tilde{C}_{111},\tilde{C}_{112},\tilde{C}_{122},\tilde{C}_{222}\big)=(2,6,6,2), \nonumber\\
\big(C_{111}^{\mathtt{f}},C_{112}^{\mathtt{f}},C_{122}^{\mathtt{f}},C_{222}^{\mathtt{f}}\big)=(-110,6,6,2).\label{eq:tCijk-Cf-P3P3}
\end{gather}
\end{Proposition}

As before, the first equality of \eqref{eq:tCijk-Cf-P3P3} is explained by mirror symmetry, i.e., the isomorphism of the B-structure at $o$ with the A-structure of $X$. To see this isomorphism more explicitly, we recall the mirror map
\begin{gather*}
t_{i}=\frac{\int_{A_{i}}\Omega_{\bm{x}}}{\int_{A_{0}}\Omega_{\bm{x}}}
\end{gather*}
defined by the B-structure at the LCSL $o$. The monodromy matrix $\mathtt{T}_{E_{1}}=(\,^{t}T_{E_{1}})^{-1}$ represents the isomorphism $H_{3}(X_{b_{o}}^{*},\mathbb{Z})\to H_{3}(X_{b_{o}'}^{*},\mathbb{Z})$ which follows from the analytic continuation of the period integral $\Pi(x,y)$ along the path $p_{b_{0}\leftarrow E_{1}\leftarrow b_{o}}$. After the continuation, the coordinate $(t_{1},t_{2})$ transformed to $(t_{1}',t_{2}')$ with
\begin{gather*}
t_{1}'=-t_{1},\qquad t_{2}'=6t_{1}+t_{2}.
\end{gather*}
 Corresponding to Proposition \ref{prop:Yukawa-Reye}, we now have
\begin{Proposition}Let $C_{ijk}$ be as defined in Proposition~{\rm \ref{prop:Nijk-N0}}. Also set $q_{1}':=e^{t_{1}'}$ and $q_{1}=e^{t_{1}}$. Then we have the following relations
\begin{gather*}
C_{ijk}^{\mathtt{f}}=\sum_{l,m,n}C_{lmn}\frac{dt_{l}}{dt_{i}'}\frac{dt_{m}}{dt_{j}'}\frac{dt_{n}}{dt_{k}'}
\end{gather*}
and
\begin{gather}
\tilde{C}_{111}+80\frac{q_{1}'}{1-q_{1}'}+4\frac{2^{3}q_{1}'^{2}}{1-q_{1}'^{2}}=C_{111}^{\mathtt{f}}+\left(80\frac{q_{1}}{1-q_{1}}+4\frac{2^{3}q_{1}^{2}}{1-q_{1}^{2}}\right)\left(\frac{dt_{1}}{dt_{1}'}\right)^{3}.\label{eq:Cijk-Flop-inv-P3P3}
\end{gather}
\end{Proposition}

In the above equality, we see the invariance of the quantum cohomology of $X$ under birational transformations. We note that the equality~(\ref{eq:Cijk-Flop-inv-P3P3}) has a slightly more general form than the familiar form~(\ref{eq:Cijk-Flop-inv}) due to the existence of~4 conics in the flopping curves.

\section{Summary and discussions}\label{section6}

We have studied gluings of monodromy nilpotent cones through monodromy relations coming from boundary divisors. Under the mirror symmetry, we have identified them with the corresponding gluings along codimension-one walls of the K\"ahler cones in birational geometry. In this paper, we confined ourself to two specific examples by doing explicit monodromy calculations. However, it is naturally expected that the observed gluings of monodromy nilpotent cones and their interpretation in mirror symmetry hold in general.

We present below some discussions and related subjects in order. In particular, we briefly report the gluing in the case of K3 surfaces whose moduli spaces have parallel structures to the Calabi--Yau threefolds $X$ and $X^{*}$ studied in Sections~\ref{sec: CICY-I} and~\ref{sec:Gluing-monod-I}.

{\bf 6.1.} The gluing of monodromy nilpotent cones has been done naturally through the monodromy relations (\ref{eq:T-relation}), (\ref{eq:T-relation2}) and also (\ref{eq:XP3P3-monod-rel-1}), (\ref{eq:XP3P3-monod-rel-2}). These relations came from boundary divisors which have tangential intersections with some component of discriminant and the blowing-ups at the intersection points. As remarked in Remarks~\ref{rem:Remark-TE1}, \ref{rem:Remark-TE1-II}, these tangential singularities are related to the contractions in the birational geometry of the mirror Calabi--Yau manifolds. We expect some generality in the degenerations of the mirror families $\mathfrak{X}^{*}$ when we approach to the exceptional divisors~$E_{i}$ of the blow-ups. We have to leave this for future investigations although we note that a categorical study of the mirror symmetry for conifold transitions has been put forward in a recent work~\cite{FHLY}.

{\bf 6.2.} In the homological mirror symmetry due to Kontsevich~\cite{Ko}, monodromy transformations in B-structures are interpreted as the corresponding transformations in the derived category of coherent sheaves $D^{b}(X)$. From this viewpoint, the gluing of nilpotent cones in $\operatorname{End}\big(H^{3}(X^{*},\mZ)\big)$ suggests the corresponding gluing of K\"ahler cones in $\operatorname{End}(K(X))$ as a \textit{homological} extension of the movable cones. The resulting wall structures of the gluing in $\operatorname{End}(K(X))$ should be regarded as the wall structures in the stability space \cite{Brid} of the objects in $D^{b}(X)$.

{\bf 6.3.} As addressed in Remark \ref{rem:Bat-Nil}, one can expect non-trivial birational geometry also for other examples of complete intersections described by Gorenstein cones~\cite{BatNil}. Among such examples, there are complete intersections whose projective geometry fits well to the so-called \textit{linear duality} (see Appendix~\ref{section:LD}). We have for example the following complete intersection:
\begin{gather*}
X=\left(\begin{matrix}\mP^{4}|\,2\,1\,1\,1\\
\mP^{3}|\,1\,1\,1\,1
\end{matrix}\right)^{2,56},
\end{gather*}
which shares many properties with (\ref{eq:XP4P4}) in Section \ref{sec: CICY-I}, e.g., three birational models come together when we construct the complete intersection of the form $X$. Although we do not have birational
automorphisms of infinite order in this example, these three birational models are explained nicely by ``double linear duality'', a certain composite of two different linear dualities. We will report this elsewhere.

{\bf 6.4 (Cayley model of Reye congruences).} Historically the Calabi--Yau complete intersection studied in Section~\ref{sec: CICY-I} is a generalization of the following K3 surface:
\begin{gather*}
X=\left(\begin{matrix}\mP^{3}\,|\,\,1\,1\,1\,1\\
\mP^{3}\,|\,\,1\,1\,1\,1
\end{matrix}\right),
\end{gather*}
which is called a Cayley model of Reye congruences. When we take the defining equations general, $X$ is a smooth K3 surface of the Picard lattice isomorphic to $M:=\left(\mZ^{2},\left(\begin{smallmatrix}4 & 6\\
6 & 4
\end{smallmatrix}\right)\right)$. This~K3 surface has been studied in \cite{FGvGvL, Oguiso2} as an example which has an automorphism $\rho$ of infinite order and also positive entropy. Actually, we have the same diagram as~(\ref{diag:XZXZ}) with the parallel definitions of $X_{i}$ ($X_{1}:=X)$ and $Z_{i}$ as well as $\rho$ in Proposition~\ref{prop:Birat(Xi)}. The difference is in that all~$X_{i}$ and~$Z_{i}$ are smooth K3 surfaces and hence
isomorphic to each other under the morphisms, e.g., $\pi_{ij}$ and $\varphi_{ij}$. For~K3 surfaces, we have the so-called counting formula~\cite{HLOY} for the number of Fourier--Mukai partners. Based on it, it is easy to see that the set $FM(X)$ of Fourier--Mukai partners consists of only $X$ itself.

The construction of the mirror family of $X$ is similar to Section~\ref{sub:MirrorFamily}, and there appear three LCSL $o_{i}$, $i=1,2,3$, on the compactified moduli space $\overline{\mathcal{M}}_{X^{*}}^{\rm cpx}=\mP^{2}$. As before, we determine the connecting matrices $\cvarphi_{ij}$ by blowing-up at three points with (fourth) tangential intersections (cf.\ Fig.~\ref{Fig2}). Making similar canonical bases of period integrals as
in (\ref{eq:PiX}) at each point, which represents bases of the transcendental lattice $T_{X^{*}}\simeq U\oplus M$ of the mirror K3 surface $X^{*}$, we obtain
\begin{gather*}
 \cvarphi_{21}=\left(\begin{smallmatrix}-1 & 0 & 0 & 0\\
0 & 1 & -3 & 0\\
0 & 0 & -1 & 0\\
0 & 0 & 0 & -1
\end{smallmatrix}\right),\qquad \cvarphi_{32}=\left(\begin{smallmatrix}-1 & 0 & 0 & 0\\
0 & -3 & 1 & 0\\
0 & -1 & 0 & 0\\
0 & 0 & 0 & -1
\end{smallmatrix}\right),\qquad \cvarphi_{13}=\left(\begin{smallmatrix}-1 & 0 & 0 & 0\\
0 & -1 & 0 & 0\\
0 & -3 & 1 & 0\\
0 & 0 & 0 & -1
\end{smallmatrix}\right),\\
\crho:=\cvarphi_{13}\cvarphi_{32}\cvarphi_{21}=\left(\begin{smallmatrix}-1 & 0 & 0 & 0\\
0 & 3 & -8 & 0\\
0 & 8 & -21 & 0\\
0 & 0 & 0 & -1
\end{smallmatrix}\right)
\end{gather*}
as elements in $O(U\oplus M,\mZ$). Here we define $U=\mZ e\oplus\mZ f$ to be the hyperbolic lattice $\left(\mZ^{2},\left(\begin{smallmatrix}0 & 1\\ 1 & 0 \end{smallmatrix}\right)\right)$ and order the bases of $U\oplus M$ as $\mZ e\oplus M\oplus\mZ f$ when writing the above matrix forms.

The classical mirror symmetry summarized in Section \ref{sec:Classical-MS} applies to the so-called (families of) lattice polarized K3 surfaces replacing the K\"ahler cone with the ample cones~\cite{Do}. In our case here, we consider a primitive lattice embedding $M\oplus U\oplus\check{M}\subset L_{K3}$ with a fixed decomposition $M^{\perp}=U\oplus\check{M}$. Then $X$ is a member of the $M$-polarized K3 surfaces, while the mirror $X^{*}$ is a member of $\check{M}$-polarized K3 surfaces (whose transcendental lattice is $\check{M}^{\perp}=U\oplus M$). The classical mirror symmetry in this case may be summarized in the following isomorphism:
\begin{gather*}
V_{M}^{+}+\sqrt{-1}M\otimes\mR\simeq\Omega^{+}(U\oplus M)
\end{gather*}
for the period domain $\Omega^{+}(U\oplus M)=\left\{ [\omega]\in\mP((U\oplus M)\otimes\mC)\,|\,\omega.\omega=0,\omega.\bar{\omega}>0\right\} ^{+}$ where we take one of the connected components, and the corresponding component of the tube domain $V_{M}^{+}=\left\{ v\in M\otimes\mR\,|\,(v,v)_{M}>0\right\} ^{+}$.

Since there are no elements with $(v,v)_{M}=-2$ in $M$, the ample cones of general members of $M$-polarized K3 surfaces coincide with the positive cone, which is isomorphic to $V_{M}^{+}$. Similarly to what we described in Section~\ref{sub:Movable-cone}, by gluing the cone $\mR_{\geq0}H_{1}+\mR_{\geq0}H_{2}\subset H^{2}(X,\mR$) by the morphisms $\varphi_{ij}$, we arrive at the positive cone $V_{M}^{+}$ which is an irrational cone (see \cite{Oguiso2} and \cite[Section~1.5]{FGvGvL}). This gluing exactly matches to the gluing the monodromy nilpotent cones at each boundary point $o_{i}$ by the connection matrix $\check{\varphi}_{ij}$. The monodromy relations play the key roles for the gluing, and they follow from the parallel calculations to those in Section~\ref{sec:Gluing-monod-I}. For example, we have
\begin{gather*}
T_{x'}=T_{x}^{-1}T_{y}^{3},\qquad T_{y'}=T_{y},\qquad T_{x''}=T_{x},\qquad T_{y''}=T_{y}^{-1}T_{x}^{3}
\end{gather*}
corresponding to (\ref{eq:T-relation}) and (\ref{eq:T-relation2}), respectively, with
\begin{gather*}
 T_{x}=\left(\begin{smallmatrix}1 & -1 & 0 & -2\\
0 & 1 & 0 & 4\\
0 & 0 & 1 & 6\\
0 & 0 & 0 & 1
\end{smallmatrix}\right),\qquad T_{y}=\left(\begin{smallmatrix}1 & 0 & -1 & -2\\
0 & 1 & 0 & 6\\
0 & 0 & 1 & 4\\
0 & 0 & 0 & 1
\end{smallmatrix}\right),\\
 T_{x'}=\cvarphi_{21}^{-1}T_{x}\cvarphi_{21},\qquad T_{y'}=\cvarphi_{21}^{-1}T_{y}\cvarphi_{21}\qquad \text{and}\qquad T_{x''}=\cvarphi_{31}^{-1}T_{x}\cvarphi_{31},\qquad T_{y''}=\cvarphi_{31}^{-1}T_{y}\cvarphi_{31}.
\end{gather*}
 As in Section~\ref{sub:Gluing-MN-subsec}, exceptional divisors $E_{1}$, $E_{1}'$ and $E_{1}''$ have to be introduced to determine the connection matrices $\cvarphi_{ij}$, but it turns out that their monodromies are trivial, i.e., $T_{E_{1}}=T_{E'}=T_{E_{1}''}=\mathrm{id}$. Clearly, this is consistent to our interpretation of these monodromies in terms of the flopping curves (Proposition~\ref{prop:Picard-Lefschetz-Reye}) for the case of Calabi--Yau threefolds.

As this example shows, irrational ample cones indicate infinite gluings of the nilpotent cones in the mirror side. It is natural to expect that the corresponding property holds for the mirror symmetry of Calabi--Yau threefolds in general with ample cones replaced by movable cones and the morphisms by birational maps as known in the so-called \textit{movable cone conjecture} \cite{KawCone, MoKahler}. We have shown in this paper that, in three dimensions, the gluings of monodromy nilpotent cones encode the non-trivial monodromies $T_{E_{i}}$ which correspond to the flopping curves.

\appendix

\section{Proof of Lemmas~\ref{lem:KahlerX2} and~\ref{lem:KahlerX3}}\label{appendixA}

\subsection{Proof of Lemma~\ref{lem:KahlerX2}}

Let us consider the projective spaces $\mP(V_{i})$ with $V_{i}\simeq\mC^{5}$, $i=1,2$. Here we will only present a~proof of~(1), but it should be clear how to modify the following setting to show~(2).

We start with our discussion with the following exact sequence, which we obtain by tensoring the Euler sequence of $\mP(V_{2})$ with $V_{1}$:
\begin{gather*}
0\to V_{1}\otimes\sO_{\mP(V_{2})}(-1)\to V_{1}\otimes V_{2}\otimes\sO_{\mP(V_{2})}\to V_{1}\otimes T_{\mP(V_{2})}(-1)\to0.
\end{gather*}
In the following arguments, we denote this sequence by
\begin{gather*}
0\to\sE\to V_{1}\otimes V_{2}\otimes\sO_{\mP(V_{2})}\to(\sE^{\perp})^{*}\to0
\end{gather*}
with defining $\sE:=V_{1}\otimes\sO_{\mP(V_{2})}(-1)$ and $\sE^{\perp}:=V_{1}^{*}\otimes\Omega_{\mP(V_{2})}(1)$. We also have the following diagram of a linear duality (cf.~\cite[Section~8]{Kuz1}):
\begin{gather*}
\begin{xy}
(-25,0)*++{X_1\;\;\subset\;\;}="Xi",
(-15,0)*++{\mP(\sE)}="PE",
(15,0)*++{\mP(\sE^\perp)}="PEp",
(25,0)*++{\quad\supset\;X_2},
(-30,-13)*++{\mP(V_1\otimes V_2)}="PVV",
( 0,-13)*++{\mP(V_2)}="PV",
( 30,-13)*++{\mP(V_1^*\otimes V_2^*)}="PVsVs",
( 45,-13)*+{\supset \; Z_3.}
\ar "PE";"PVV"
\ar "PE";"PV"
\ar "PEp";"PV"
\ar "PEp";"PVsVs"
\end{xy}
\end{gather*}
Note that $\mP(\sE)$ is isomorphic to $\mP(V_{1})\times\mP(V_{2})$, and $\sO_{\mP(\sE)}(1)\simeq\sO_{\mP(V_{1})\times\mP(V_{2})}(1,1)$ since it is the pull-back of $\sO_{\mP(V_{1}\otimes V_{2})}(1)$ by construction. Therefore $X_{1}$ is a codimension 5 complete intersection in $\mP(\sE)$ with respect to $\sO_{\mP(\sE)}(1)$, and we have $\sO_{\mP(\sE)}(1)|_{X_{1}}=H_{1}+H_{2}$.

We see that
\begin{gather}
\mP(\sE^{\perp})=\{(w,M)\,|\, Mw=0\}\subset\mP(V_{2})\times\mP(V_{1}^{*}\otimes V_{2}^{*}),\label{eq:springer-resol-Pep}
\end{gather}
where we consider $V_{1}^{*}\otimes V_{2}^{*}\simeq\Hom(V_{2},V_{1}^{*})$ and $M$ is a $5\times5$ matrix. Therefore the image~$\sZ$ of the map $\mP(\sE^{\perp})\to\mP(V_{1}^{*}\otimes V_{2}^{*})$ consists of $5\times5$ matrices of rank $\leq4$, thus $\sZ$ is so-called the determinantal quintic. Note that we can write the determinantal quintic $Z_{3}\subset\mathbb{P}_{\lambda}^{4}$ in Proposition~\ref{proposition3.3} by $Z_{3}=\sZ\cap P_{4}$ for a $4$-dimensional linear subspace $P_{4}\subset\mP(V_{1}^{*}\otimes V_{2}^{*})$ with identifying~$P_{4}$ with~$\mathbb{P}_{\lambda}^{4}$. Moreover, the pull-back of~$Z_{3}$ to~$\mP(\sE^{\perp})$ is~$X_{2}$.

By a general fact on linear duality (\ref{eq:appendixB-H-H-det}) in Appendix~\ref{section:LD}, we have
\begin{gather}
\sO_{\mP(\sE)}(1)|_{X_{1}}+\sO_{\mP(\sE^{\perp})}(1)|_{X_{1}}=\det\sE^{*}=5H_{2},\label{eq:appendixA-H-H-det}
\end{gather}
where we denote by $\sO_{\mP(\sE^{\perp})}(1)|_{X_{1}}$ the strict transform of $\sO_{\mP(\sE^{\perp})}(1)|_{X_{2}}$ and abbreviate the notation for the pull-back for $\det\sE^{*}$. In this appendix, unless stated otherwise, we will write proper transforms of a divisor by the same symbol omitting the pull-backs by birational maps. Using this convention, we have $\sO_{\mP(\sE)}(1)|_{X_{1}}=H_{1}+H_{2}$ and also $\sO_{\mP(\sE^{\perp})}(1)|_{X_{1}}=L_{Z_{3}}$. Then we have
\begin{gather*}
(H_{1}+H_{2})+L_{Z_{3}}=5H_{2},
\end{gather*}
which gives $L_{Z_{3}}=4H_{2}-H_{1}$. Therefore, restoring the pull-backs by birational maps, we have
\begin{gather*}
\varphi_{21}^{*}L_{Z_{3}}=4H_{2}-H_{1},\qquad \varphi_{21}^{*}L_{Z_{2}}=H_{2}
\end{gather*}
in $N^{1}(X)$, which determine $\varphi_{21}^{*}(\sK_{X_{2}})$ as claimed.

\subsection{Proof of Lemma~\ref{lem:KahlerX3}}\label{appendixA.2}

Basic idea is very similar to the linear duality in the previous section. We consider the following diagram:
\begin{gather*}
\begin{xy}
(-25,0)*++{\mP(V_1^*\otimes\Omega_{\mP(V_2)}(1))}="PVO",
( 0,-13)*++{\mP(V_1^*\otimes V_2^*).}="PVsVs",
( 25,0)*++{\mP(\Omega_{\mP(V_1)}(1)\otimes V_2^*)}="PVOt"
\ar "PVO";"PVsVs"
\ar "PVOt";"PVsVs"
\end{xy}
\end{gather*}

\begin{cla}$\mP(V_{1}^{*}\otimes\Omega_{\mP(V_{2})}(1))\to\mP(V_{1}^{*}\otimes V_{2}^{*})$ and $\mP(\Omega_{\mP(V_{1})}(1)\otimes V_{2}^{*})\to\mP(V_{1}^{*}\otimes V_{2}^{*})$ are flopping contractions onto the common image~$\sZ$. Moreover, it is of Atiyah type outside the locus in $\sZ$ of corank $\geq2$.
\end{cla}
\begin{proof} This is standard since $\mP(V_{1}^{*}\otimes\Omega_{\mP(V_{2})}(1))\to\mP(V_{1}^{*}\otimes V_{2}^{*})$ and $\mP(\Omega_{\mP(V_{1})}(1)\otimes V_{2}^{*})\to\mP(V_{1}^{*}\otimes V_{2}^{*})$ are the Springer type resolutions of the image $\mathcal{Z}$ (see~(\ref{eq:springer-resol-Pep})).
\end{proof}

As we have seen in the proof of Lemma~\ref{lem:KahlerX2}, $X_{2}$ is contained in
$\mP(V_{2}^{*}\otimes\Omega_{\mP(V_{1})}(1))$. Similarly, $X_{3}$
is contained in $\mP(\Omega_{\mP(V_{2})}(1)\otimes V_{1}^{*})$. Indeed,
for the 4-dimensional linear subspace $P_{4}\subset\mP(V_{1}^{*}\otimes V_{2}^{*})$
such that $Z_{3}=\sZ\cap P_{4}$, $X_{2}$ and $X_{3}$ are the pull-backs
of $P_{4}$ to $\mP(V_{1}^{*}\otimes\Omega_{\mP(V_{2})}(1))$ and
$\mP(\Omega_{\mP(V_{1})}(1)\otimes V_{2}^{*})$, respectively.

Now we take the fiber product
\begin{gather*}
\mP:=\mP\big(V_{1}^{*}\otimes\Omega_{\mP(V_{2})}(1)\big)\times_{\mP(V_{1}^{*}\otimes V_{2}^{*})}\mP\big(\Omega_{\mP(V_{1})}(1)\otimes V_{2}^{*}\big).
\end{gather*}
\begin{cla}
It holds that $\mP=\mP_{\mP(V_{1})\times\mP(V_{2})}\big(\Omega_{\mP(V_{1})}(1)\boxtimes\Omega_{\mP(V_{2})}(1)\big)$.
\end{cla}
\begin{proof}Note that
\begin{gather*}
\mP=\big\{(w,M,z)\,|\, Mw=0,{\empty^{t}z}M=0\big\}\subset\mP(V_{2})\times\mP(V_{1}^{*}\otimes V_{2}^{*})\times\mP(V_{1}).
\end{gather*}
Thus the fiber of $\mP\to\mP(V_{1})\times\mP(V_{2})$ over $(w,z)$ is nothing but $\mP((V_{1}/\mC w)^{*}\otimes(V_{2}/\mC z)^{*})$ and the assertion follows.
\end{proof}

Note that the tautological divisor $\sO_{\mP}(1)$ of $\mP\big(\Omega_{\mP(V_{1})}(1)\boxtimes\Omega_{\mP(V_{2})}(1)\big)$ defines a map to $\mP(V_{1}^{*}\otimes V_{2}^{*})$ and it is the pull-back of $\sO_{\mP(V_{1}^{*}\otimes V_{2}^{*})}(1)$. We will denote it by $L_{\mP(V_{1}^{*}\otimes V_{2}^{*})}$. By the canonical bundle formula of $\mP(\Omega_{\mP(V_{1})}(1)\boxtimes\Omega_{\mP(V_{2})}(1))$, we have
\begin{gather*}
K_{\mP}=-16L_{\mP(V_{1}^{*}\otimes V_{2}^{*})}+K_{\mP(V_{1})\times\mP(V_{2})}+\det\big\{\Omega_{\mP(V_{1})}(1)\boxtimes\Omega_{\mP(V_{2})}(1)\big\}^{*},
\end{gather*}
where we omit the notation of the pull-backs for $K_{\mP(V_{1})\times\mP(V_{2})}$ and $\det\{\Omega_{\mP(V_{1})}(1)\boxtimes\Omega_{\mP(V_{2})}(1)\}^{*}$. Since $K_{\mP(V_{1})\times\mP(V_{2})}=-5L_{\mP(V_{1})}-5L_{\mP(V_{2})}$, where $L_{\mP(V_{1})}$ and $L_{\mP(V_{2})}$ are the pull-backs of $\sO_{\mP(V_{i})}(1)$'s of $\mP(V_{i})$ on the left and right factors of $\mP(V_{1})\times\mP(V_{2})$, respectively, and $\det\{\Omega_{\mP(V_{1})}(1)\boxtimes\Omega_{\mP(V_{2})}(1)\}^{*}$ $=4L_{\mP(V_{1})}+4L_{\mP(V_{2})}$, we have
\begin{gather}
K_{\mP}=-16L_{\mP(V_{1}^{*}\otimes V_{2}^{*})}-L_{\mP(V_{1})}-L_{\mP(V_{2})}.\label{eq:App-KmP}
\end{gather}
By the canonical bundle formula of $\mP(V_{1}^{*}\otimes\Omega_{\mP(V_{2})}(1))$, we have
\begin{gather*}
-K_{\mP(V_{1}^{*}\otimes\Omega_{\mP(V_{2})}(1))}=20L_{\mP(V_{1}^{*}\otimes V_{2}^{*})}.
\end{gather*}
Pushing forwards (\ref{eq:App-KmP}) to $\mP(V_{1}^{*}\otimes\Omega_{\mP(V_{2})}(1))$, we obtain
\begin{gather*}
-K_{\mP(V_{1}^{*}\otimes\Omega_{\mP(V_{2})}(1))}=16L_{\mP(V_{1}^{*}\otimes V_{2}^{*})}+L_{\mP(V_{1})}+L_{\mP(V_{2})}.
\end{gather*}
Therefore we have
\begin{gather}
L_{\mP(V_{1})}+L_{\mP(V_{2})}=4L_{\mP(V_{1}^{*}\otimes V_{2}^{*})}.\label{eq:new706}
\end{gather}

Now, restricting the above construction over the linear subspace $P_{4}\subset\mP(V_{1}^{*}\otimes V_{2}^{*})$, we have
\begin{gather*}
\begin{xy}
(0,0)*++{\mP_{|P_4}}="topP",
(-15,-10)*++{X_2}="Xl",
( 15,-10)*++{X_3,}="Xr",
( 0,-20)*++{P_4}="botP"
\ar "topP";"Xl"
\ar "topP";"Xr"
\ar "Xl";"botP"
\ar "Xr";"botP"
\end{xy}
\end{gather*}
where we denote by $\mP_{|P_{4}}$ the restriction of $\mP$ over $P_{4}$. Restricting (\ref{eq:new706}) to $X_{2}$, we have
\begin{gather}
\varphi_{32}^{*}(M_{Z_{1}})+L_{Z_{2}}=4L_{Z_{3}}.\label{eq:desiredrel}
\end{gather}
This is the claimed relation.
\begin{Corollary}$\mP(V_{1}^{*}\otimes\Omega_{\mP(V_{2})}(1))\dashrightarrow\mP(\Omega_{\mP(V_{1})}(1)\otimes V_{2}^{*})$ is the flop. Similarly, $X_{2}\dashrightarrow X_{3}$ is the flop. \end{Corollary}
\begin{proof}
Note that $L_{\mP(V_{1})}$ and $L_{\mP(V_{2})}$ are relatively ample for $\mP(\Omega_{\mP(V_{1})}(1)\otimes V_{2}^{*})\to\mP(V_{1})$ and $\mP(V_{1}^{*}\otimes\Omega_{\mP(V_{2})}(1))\to\mP(V_{2})$, respectively. Since $L_{\mP(V_{1}^{*}\otimes V_{2}^{*})}$ is the pull-backs of a divisor on $\mP(V_{1}^{*}\otimes V_{2}^{*})$, we see that $-L_{\mP(V_{1})}$ is relatively ample for $\mP(V_{1}^{*}\otimes\Omega_{\mP(V_{2})}(1))\to\mP(V_{2})$ by (\ref{eq:new706}). Therefore $\mP(V_{1}^{*}\otimes\Omega_{\mP(V_{2})}(1))\dashrightarrow\mP(\Omega_{\mP(V_{1})}(1)\otimes V_{2}^{*})$ is the flop. We can show the assertion for $X_{2}\dashrightarrow X_{3}$ in the same way using (\ref{eq:desiredrel}).
\end{proof}

\section{Linear duality}\label{section:LD}

Having the case $W=V_{1}\otimes V_{2}$ and $B=\mP(V_{2})$ in mind, we consider the exact sequence of sheaves (vector bundles) in the following general form with $\dim W=N$:
\begin{gather*}
0\to\sE\to W\otimes\sO_{B}\to\big(\sE^{\perp}\big)^{*}\to0,\\
0\to\sE^{\perp}\to W^{*}\otimes\sO_{B}\to\sE^{*}\to0.
\end{gather*}
Under this general setting, we have the following natural morphisms:
\begin{gather*}
\begin{xy}
(-15,0)*++{\mP(\sE)}="PE",
( 15,0)*++{\mP\big(\sE^{\perp}\big)}="PEp",
(-30,-13)*++{\mP(W)}="PW",
( 0,-13)*++{B}="B",
( 30,-13)*++{\mP(W^{*}).}="PWs",
\ar_{f} "PE";"PW"
\ar"PE";"B"
\ar"PEp";"B"
\ar^{g} "PEp";"PWs"
\end{xy}
\end{gather*}

\begin{Lemma}\label{lem:appendix-dim} Let $\sE_{b}$ and $\sE_{b}^{\perp}$ be the fibers over $b\in B$ of $\sE$ and $\sE^{\perp}$, respectively. Then it holds
\begin{gather*}
\dim\mP\big(\sE_{b}\cap L_{r}^{\perp}\big)=\dim\mP\big(\sE_{b}^{\perp}\cap L_{r}\big)
\end{gather*}
for any $r$-dimensional linear subspace $L_{r}\subset W^{*}$ and the orthogonal linear subspace $L_{r}^{\perp}$ in $W$. \end{Lemma}
\begin{proof}We calculate the dimensions as follows: $\dim\big(\sE_{b}\cap L_{r}^{\perp}\big)=\dim\sE_{b}+\dim L_{r}^{\perp}-\dim\big(\sE_{b}+L_{r}^{\perp}\big)=r+(N-r)-\dim\big(\sE_{b}+L_{r}^{\perp}\big)=\dim\big(\sE_{b}^{\perp}\cap L_{r}\big)$.
\end{proof}

The complete intersections $X_{1}$, $X_{2}$ in Appendix~\ref{appendixA.2} may be described, respectively, in general terms as
\begin{gather*}
X_{L_{r}^{\perp}}=f^{-1}\big(L_{r}^{\perp}\big)\cap\mP(\sE),\qquad Y_{L_{r}}=g^{-1}(L_{r})\cap\mP\big(\sE^{\perp}\big)
\end{gather*}
for a fixed subspace $L_{r}\subset W^{*}$, which we call \textit{orthogonal linear sections.} Consider the Grassmannian $\rG=\operatorname{Gr}(r,N)$ of $r$-spaces in $W^{*}$ and define the following family of orthogonal
linear sections:
\begin{gather*}
 \mathcal{X}_{r}:=\big\{ ([L_{r}],x)\in\rG\times\mP(\sE)\,|\, f(x)\in\mP\big(L_{r}^{\perp}\big)\big\} ,\\
 \mathcal{Y}_{r}:=\big\{ ([L_{r}],y)\in\rG\times\mP(\sE^{\perp})\,|\, g(y)\in\mP(L_{r})\big\} .
\end{gather*}
Also we define
\begin{gather*}
\Sigma_{0}:=\big\{ ([L_{r}],b)\in\rG\times B\,|\,\sE_{b}\cap L_{r}^{\perp}\not=0\big\} =\big\{ ([L_{r}],b)\in\rG\times B\,|\,\sE_{b}^{\perp}\cap L_{r}\not=0\big\} ,
\end{gather*}
where the second equality is valid due to Lemma~\ref{lem:appendix-dim}. Then $\mathcal{X}_{r}$, $\mathcal{Y}_{r}$ are orthogonal linear sections fibered over $\Sigma_{0}$ and, with natural morphisms, they
can be arranged in the following diagram:
\begin{gather}
\begin{xy}
( 0,13)*++{\mathcal{X}_{r}\times_{\Sigma_{0}}\mathcal{Y}_{r}}="XrYr",
(-15,0)*++{\mathcal{X}_{r}}="Xr",
( 15,0)*++{\mathcal{Y}_{r}}="Yr",
( 0,-13)*++{\Sigma_{0}}="Sigma",
(-30,0)*++{\mP(\sE)}="PE",
(30,0)*++{\mP\big(\sE^{\perp}\big).}="PEp",
( 12,-13)*{\subset\;\rG\times B},
\ar(-6,9.5);"Xr"
\ar(6,9.5);"Yr"
\ar"Xr";"Sigma"
\ar"Yr";"Sigma"
\ar"Xr";"PE"
\ar"Yr";"PEp"
\end{xy}\label{eq:appendixB-diagram}
\end{gather}
Let us introduce the following divisors related to the diagram:
\begin{gather*}
H_{\sE}:=\sO_{\mP(\sE)}(1),\qquad H_{\sE^{\perp}}:=\sO_{\mP(\sE^{\perp})}(1),\qquad H_{\rG}:=\sO_{\rG}(1).
\end{gather*}

\begin{Proposition}\label{prop:appendixB-KKK} Abbreviating the pull-back symbols by the morphisms in the diagram \eqref{eq:appendixB-diagram}, we have
\begin{gather*}
K_{\mathcal{X}_{r}\times_{\Sigma_{0}}\mathcal{Y}_{r}}=-(N-2)H_{\rG}-H_{\sE}-H_{\sE^{\perp}}+K_{B}+2\det\,\sE^{*}
\end{gather*}
and
\begin{gather*}
 K_{\mathcal{X}_{r}}=-(N-1)H_{\rG}+K_{B}+\det \sE^{*},\\
 K_{\mathcal{Y}_{r}}=-(N-1)H_{\rG}+K_{B}+\det \big(\sE^{\perp}\big)^{*}.
\end{gather*}
\end{Proposition}
\begin{proof}We leave the proofs for readers.
\end{proof}

It is easy to recognize that the proofs of the above proposition rely on the projective geometry behind the diagram (\ref{eq:appendixB-diagram}). We will report the proofs elsewhere with some additional properties which we can extract from the diagram (\ref{eq:appendixB-diagram}); for example, we can show that the morphisms $\mathcal{X}_{r}\to\Sigma_{0}$, $\mathcal{Y}_{r}\to\Sigma_{0}$ are flopping contractions and the naturally induced birational map $\mathcal{X}_{r}\dashrightarrow\mathcal{Y}_{r}$ in the diagram is the flop for these contractions.
\begin{Proposition}\label{prop:appendixB-KKK2}Pushing forward $K_{\mathcal{X}_{r}\times_{\Sigma_{0}}\mathcal{Y}_{r}}$ to $\mathcal{X}_{r}$, and equating to$K_{\mathcal{X}_{r}}$, we have a relation
\begin{gather}
H_{\sE}+H_{\sE^{\perp}}=\det \sE^{*}+H_{\rG}\label{eq:appendixB-H-relation-Xr}
\end{gather}
on $\mathcal{X}_{r}$. Similarly, we have a corresponding relation on $\mathcal{Y}_{r}$,
\begin{gather*}
H_{\sE}+H_{\sE^{\perp}}=\det \big(\sE^{\perp}\big)^{*}+H_{\rG}.
\end{gather*}
\end{Proposition}

Now restricting the relation (\ref{eq:appendixB-H-relation-Xr}) on~$\mathcal{X}_{r}$ to $\mathcal{X}_{r}|_{[L_{r}]\times\mP(\sE)}=X_{L_{r}^{\perp}}$, we obtain
\begin{gather}
H_{\sE}|_{X_{L_{r}^{\perp}}}+H_{\sE^{\perp}}|_{X_{L_{r}}^{\perp}}=\det\,\sE^{*},\label{eq:appendixB-H-H-det}
\end{gather}
which is the relation we used in~(\ref{eq:appendixA-H-H-det}).

\subsection*{Acknowledgements} The results of this work have been reported by the first named author (S.H.) in several workshops; ``Modular forms in string theory'' at BIRS (2016), ``Categorical and Analytic invariants~IV'' at Kavli IPMU (2016), ``Workshop on mirror symmetry and related topics'' at Kyoto University (2016) and ``The 99th Encounter between Mathematicians and Theoretical Physicists'' at IRMA, Strasbourg (2017). He would like to thank the organizers for the invitations for the workshops where he had valuable discussions with the participants. Writing this paper started when S.H.~was staying at Brandeis University and Harvard University in March, 2017. He would like to thank B.~Lian and S.-T.~Yau for their kind hospitality and also valuable discussions during his stay. The authors would like to thank anonymous referees for valuable comments which helped them improve this paper. This work is supported in part by Grant-in Aid Scientific Research JSPS (C~16K05105, JP17H06127 S.H. and C 16K05090 H.T.).

\pdfbookmark[1]{References}{ref}
\LastPageEnding

\end{document}